\documentclass[reqno]{amsart}     
\AtBeginDocument{\def\MR#1{}}

\usepackage{amsmath,amssymb,amsaddr,bbm}    
\usepackage{enumerate}

\usepackage[colorlinks]{hyperref}

\newcommand{\macrocolour}{\color{black}}

\newtheorem{theorem}                   {Theorem}
 
\newtheorem{lemma}           [theorem] {Lemma}  
   
\newtheorem{corollary}       [theorem] {Corollary}

\newtheorem{prop}            [theorem] {Proposition}  
\newtheorem{claim}           [theorem] {Claim}

\newtheorem{definition}      [theorem] {Definition}

\theoremstyle{remark}
\newtheorem{remark}          [theorem] {Remark}

\newcommand{\eps}{\varepsilon}
\newcommand{\NN}{\mathbb{N}}
\newcommand{\ZZ}{\mathbb{Z}}
\newcommand{\RR}{\mathbb{R}}
\newcommand{\EE}{\mathbb{E}}
\newcommand{\VV}{\mathbb{V}}
\newcommand{\PP}{\mathbb{P}}

\newcommand{\cA}{\mathcal{A}}
\newcommand{\cB}{\mathcal{B}}
\newcommand{\cC}{\mathcal{C}}
\newcommand{\cD}{\mathcal{D}}
\newcommand{\cF}{\mathcal{F}}
\newcommand{\cG}{\mathcal{G}}

\newcommand{\cL}{\mathcal{L}}
\newcommand{\cP}{\mathcal{P}}

\newcommand{\cM}{\mathcal{M}}
\newcommand{\cN}{\mathcal{N}}
\newcommand{\cS}{\mathcal{S}}
\newcommand{\cT}{\mathcal{T}}
\newcommand{\cY}{\mathcal{Y}}

\newcommand{\No}{\mathrm{N}}
\newcommand{\Po}{\mathrm{Po}}

\newcommand{\Be}{\mathrm{Be}}

\newcommand{\Ind}[1]{\mathbbm{1}_{#1}}
\newcommand{\Indens}[1]{\Ind{\{#1\}}}

\newcommand{\cond}{\; \middle\vert \;}

\newcommand{\harm}[1]{{\macrocolour H_{#1}}}

\newcommand{\RVnumberAP}[2]{{\macrocolour X_{#1}^{#2}}}
\newcommand{\RVnumberAPCentralised}[2]{{\macrocolour \bar{X}_{#1}^{#2}}}

\newcommand{\IndicatorRVAPCentralised}[1]{{\macrocolour Z_{#1}}}

\newcommand{\reordering}{{\macrocolour \pi}}
\newcommand{\typeFunction}{{\macrocolour \tau}}
\newcommand{\type}{{\macrocolour t}}
\newcommand{\typeVect}{{\macrocolour \bf \type}}

\newcommand{\typeSizeVectSpace}[2]{{\macrocolour \cS_{#1}(#2)}}
\newcommand{\typeVectSpaceValid}[1]{{\macrocolour \cT_{#1}}}
\newcommand{\IndexSet}[1]{{\macrocolour I_{#1}}}
\newcommand{\IndexSetSize}[1]{{\macrocolour \left|\IndexSet{#1}\right|}}
\newcommand{\IndexSetMixed}[2]{{\macrocolour \IndexSet{#1,#2}}}
\newcommand{\IndexSetSizeMixed}[2]{{\macrocolour \IndexSetSize{#1,#2}}}
\newcommand{\typeSize}{{\macrocolour s}}
\newcommand{\typeSizeVect}{{\macrocolour \bf \typeSize}}

\newcommand{\match}{{\macrocolour \nu}}
\newcommand{\mainContribution}[1]{{\macrocolour F(#1)}}
\newcommand{\mainContributionSet}[1]{{\macrocolour \cF(#1)}}
\newcommand{\mainContributionSetMatching}[2]{{\macrocolour \cF_{#2}(#1)}}

\newcommand{\minorContribution}[1]{{\macrocolour G(#1)}}
\newcommand{\minorContributionSet}[1]{{\macrocolour \cG(#1)}}

\newcommand{\contribution}[1]{{\macrocolour \mu_{#1}}}
\newcommand{\numberAPTuples}[1]{{\macrocolour M_{#1}}}
\newcommand{\numberMixedAPTuples}[1]{{\macrocolour \numberAPTuples{#1}}}

\newcommand{\setAP}[2]{{\macrocolour \cA_{#1}^{#2}}}
\newcommand{\numberAP}[1]{{\macrocolour A_{#1}}}

\newcommand{\setLoosePairs}[1]{{\macrocolour \cB_{#1}}}
\newcommand{\setMixedLoosePairs}[2]{{\macrocolour \setLoosePairs{#1,#2}}}
\newcommand{\setOverlapPairs}[1]{{\macrocolour \cC_{#1}}}
\newcommand{\setMixedOverlapPairs}[2]{{\macrocolour \setOverlapPairs{#1,#2}}}
\newcommand{\numberLoosePairs}[1]{{\macrocolour B_{#1}}}
\newcommand{\numberMixedLoosePairs}[2]{{\macrocolour \numberLoosePairs{#1,#2}}}
\newcommand{\numberLoosePairsLeadingConstant}[1]{{\macrocolour \beta_{#1}}}
\newcommand{\numberMixedLoosePairsLeadingConstant}[2]{{\macrocolour\numberLoosePairsLeadingConstant{#1,#2}}}
\newcommand{\numberPartitionPairs}[2]{{\macrocolour D_{#1}^{(#2)}}}
\newcommand{\numberMixedPartitionPairs}[3]{{\macrocolour \numberPartitionPairs{#1,#2}{#3}}}

\newcommand{\setPartitionPairs}[2]{{\macrocolour \cD_{#1}^{(#2)}}}
\newcommand{\setMixedPartitionPairs}[3]{{\macrocolour \setPartitionPairs{#1,#2}{#3}}}

\newcommand{\numberOverlapPairs}[1]{{\macrocolour C_{#1}}}
\newcommand{\numberMixedOverlapPairs}[2]{{\macrocolour \numberOverlapPairs{#1,#2}}}

\newcommand{\depGraph}[1]{\cG_{#1}}

\newcommand{\neighbourhood}[2]{\cN_{#2}(#1)}
\newcommand{\neighbourhoodClosed}[2]{\widehat{\cN}_{#2}(#1)}
\newcommand{\SteinControl}[2]{{\macrocolour \mathcal{V}_{#1}(#2)}}

\newcommand{\covarianceLeadingConstant}[2]{{\macrocolour \kappa_{#1,#2}}}
\newcommand{\covarianceFunction}[1]{{\macrocolour h_{#1}}}

\newcommand{\totalOrder}{{\macrocolour \pi}}

\newcommand{\pos}[1]{{\macrocolour \iota_{#1}}}
\newcommand{\posOne}{{\macrocolour \iota}}
\newcommand{\posTwo}{{\macrocolour \iota'}}
\newcommand{\comDiff}[1]{{\macrocolour \delta_{#1}}}
\newcommand{\comDiffOne}{{\macrocolour \delta}}
\newcommand{\comDiffTwo}{{\macrocolour \delta'}}
\newcommand{\comDiffMaxOne}{{\macrocolour \Delta}}
\newcommand{\comDiffMaxTwo}{{\macrocolour \Delta'}}
\newcommand{\interPoint}{{\macrocolour m}}

\newcommand{\posRescaled}[1]{{\macrocolour a_{#1}}}
\newcommand{\posRescaledOne}{{\macrocolour a}}
\newcommand{\posRescaledTwo}{{\macrocolour a'}}

\newcommand{\comDiffRescaledOne}{{\macrocolour u}}
\newcommand{\comDiffRescaledTwo}{{\macrocolour u'}}
\newcommand{\measureComDiff}[1]{{\macrocolour \nu_{#1}}}
\newcommand{\measurePos}[1]{{\macrocolour \mu_{#1}}}
\newcommand{\ellOne}{{\macrocolour  \ell}}
\newcommand{\ellTwo}{{\macrocolour  \ell'}}

\newcommand{\Gaussian}[5]{{\macrocolour \No\left(\meanVector{#1}{#2},\covarianceMatrix{#3}{#4}{#5}\right)}}
\newcommand{\covarianceMatrix}[3]{{\macrocolour \begin{pmatrix}#1&#2\\#2&#3\end{pmatrix}}}
\newcommand{\meanVector}[2]{{\macrocolour \begin{pmatrix}#1\\#2\end{pmatrix}}}

\def\longlongrightarrow{\hspace{+0.1ex} - \hspace{-1.1ex} - \hspace{-1.1ex} - \hspace{-1.1ex}\longrightarrow  } 
\newcommand{\tendsto}[2]{ \underset{#1 \rightarrow #2}{\longlongrightarrow} }

\begin{document}
\title[Arithmetic progressions in random sets]{Bivariate fluctuations for the number of arithmetic progressions in random sets}

\author{Yacine Barhoumi-Andr\'eani}
\address{Ruhr-Universit\"at Bochum, Fakult\"at f\"ur Mathematik, Universit\"atsstrasse 150, 44780 Bochum, Germany.}
\email{\tt yacine.barhoumi@rub.de}
\author{Christoph Koch}
\address{Department of Statistics, University of Oxford, St. Giles 24-29, Oxford OX1 3LB, UK.}
\email{\tt christoph.koch@stats.ox.ac.uk}
\author{Hong Liu}
\address{Mathematics Institute and DIMAP, University of Warwick, Coventry, CV4 7AL, UK.}
\email{\tt h.liu.9@warwick.ac.uk}
\thanks{Y.B. was supported by EPSRC grant No.\ EP/L012154/1. C.K.\ was supported by EPSRC Grant No.\ EP/N004833/1 and ERC Grant No.\ 639046. H.L.\ was supported by the Leverhulme Trust Early Career Fellowship~ECF-2016-523.}
\date{\today}

\begin{abstract}
We study arithmetic progressions $\{a,a+b,a+2b,\dots,a+(\ell-1) b\}$, with $\ell\ge 3$, in random subsets of the initial segment of natural numbers $[n]:=\{1,2,\dots, n\}$. Given $p\in[0,1]$ we denote by $[n]_p$ the random subset of $[n]$ which includes every number with probability $p$, independently of one another. The focus lies on sparse random subsets, i.e.\ when $p=p(n)=o(1)$ as $n\to+\infty$.

Let $X_\ell$ denote the number of distinct arithmetic progressions of length $\ell$ which are contained in $[n]_p$. We determine the limiting distribution for $X_\ell$ not only for fixed $\ell\ge 3$ but also when  $\ell=\ell(n)\to+\infty$. The main result concerns the joint distribution of the pair $(X_{\ell},X_{\ell'})$, $\ell>\ell'$, for which we prove a bivariate central limit theorem for a wide range of $p$. Interestingly, the question of whether the limiting distribution is trivial, degenerate, or non-trivial is characterised by the asymptotic behaviour (as $n\to+\infty$) of the threshold function $\psi_\ell=\psi_\ell(n):=np^{\ell-1}\ell$. The proofs are based on the method of moments  and combinatorial arguments, such as an algorithmic enumeration of collections of arithmetic progressions. 
\end{abstract}
\maketitle
%%%%%%%%%%%%%%%%%%%%%%%%%%%%%%%%%%%%%%%%%%%%%%%%%%%%%
%%%%%%%%%%%%%%%%%%%%%%%%%%%%%%%%%%%%%%%%%%%%%%%%%%%%%

\section{Introduction and main results}\label{Sec:Introduction} 

An $\ell$-term arithmetic progression ($ \ell $-AP) in a set $ \mathcal{X}\subset \ZZ $ is an (ordered) $ \ell $-tuple of distinct numbers $ (a, a + b, \dots, a + (\ell - 1)b) $ whose elements belong to $ \mathcal{X} $. In Dickson's \emph{History of the Theory of Numbers}, the analysis of APs is traced back to around 1770 when it became prominent due to Lagrange and Waring investigating how large the common difference of an $ \ell $-AP of primes must be. Ever since, the study of APs has remained an extremely active domain of research and led to several results of fundamental importance, for instance Dirichlet's Theorem~\cite{Dirichlet} proved in 1837 played a key role in the formation of analytic number theory. Perhaps unsurprisingly, APs also became objects of interest in other fields such as combinatorics: Erd\H{o}s stated a number of conjectures related to $\ell$-APs~\cite[pp. 232-233]{Bollobas}. In particular, he offered \$1000 to solve the following \textit{largest progression-free subset} problem: find the cardinality of the largest subset of $ \{1, \dots, m\} $ ($m\in\NN$) which does not contain any $\ell $-AP. This problem was solved by Sz\'emeredi  with his celebrated density theorem~\cite{Szemeredi}: a subset of $\NN$ of non-zero upper asymptotic density contains $\ell$-APs of any arbitrary length $\ell$. Subsequently, based on Sz\'emeredi's Theorem, Green and Tao~\cite{GreenTaoPrimes} proved the long-standing conjecture on prime APs: (dense subsets of) the primes contain infinitely many $\ell$-APs for all lengths $\ell$. 

In 1936, Cram\'er~\cite{Cramer} conjectured that the gaps between two consecutive primes remain asymptotically bounded by the square of their logarithms and backed this conjecture with a heuristic model that replaces the set $ \mathcal{P} $ of primes by a random set $\mathcal{P}'$ made out of Bernoulli random variables, where $\PP(m\in\mathcal{P}')\approx 1/\log m$ independently for all integers $m\ge2$. However, the study of APs in random sets does not only provide a nice heuristic for number theoretic problems but is also a very natural and interesting model from a probabilistic point of view. For instance, Kohayakawa, \L uczak, and R\"odl~\cite{KLR} proved that  \emph{sparse} uniformly random subsets $M\subseteq\{1,\dots,n\}$ of size $|M|=\Omega(\sqrt{n})$ have the property that any (sufficiently) dense subset of $M$ already contains a $3$-AP with probability tending to $1$ as $n\to+\infty$.

In this article we focus our attention on longer APs in sparse binomial subsets of $\{1,\dots,n\}$, including $\ell$-APs with length $\ell=\ell(n)\to+\infty$ as $n\to+\infty$. In particular, we determine the limiting distribution of the number of $\ell$-APs and analyse the joint distribution of the numbers of $\ell$-APs and $\ell'$-APs of different lengths $\ell\neq\ell'$.

\subsection{Main results}
We consider a family of random subsets of the initial segments $[n]:=\{1,\dots,n\}\subset\NN$ of the integers. For any $p=p(n)\in [0,1]$ let $\Xi_1,\dots,\Xi_n$ be a collection of independent identically distributed $\Be(p)$ random variables, denote their product measure by $\PP$, and let $[n]_p:=\{i\in[n]\colon \Xi_i=1\}$ be the $p$-percolation of $[n]$, i.e.\ $[n]_p$ is the random subset of $[n]$ obtained by deleting any of the elements with probability $1-p$, independently of all other elements. We use the term \emph{constant} to mean independent of the parameter $n$, and any unspecified asymptotic notation (including limits) is to be understood with respect to $n\to+\infty$.

For any  integer $ \ell \in \{ 3, \dots, n \} $ we denote the set of all $\ell$-APs in $[n]$ by $\setAP{\ell}{} $ and define $\RVnumberAP{\ell}{}$ to be the random variable counting the number of $\ell$-APs in $[n]_p$, namely
$$
\RVnumberAP{\ell}{}=\RVnumberAP{\ell}{}(n):=|\setAP{\ell}{}|=\sum_{T\in\setAP{\ell}{}}\Ind{\{T\subseteq[n]_p\}}.
$$
Clearly, $[n]$ itself is an $n$-AP and any $\ell$-AP contains a whole number of $\ell'$-APs for each $3\le \ell'\le\ell-1$. Therefore, the family $\{\RVnumberAP{\ell}{}\}_{3\le\ell\le n}$ is obviously correlated in a non-trivial way. While the FKG inequality (e.g.\ Theorem~2.12 in \cite{JansonLuczakRucinski}) implies that this family is actually positively correlated, it is a priori unclear whether this correlation is asymptotically relevant. The main goal of this article is to study the asymptotic behaviour of the joint distribution of the pair $(\RVnumberAP{\ell_1}{},\RVnumberAP{\ell_2}{})$ with $\ell_1>\ell_2$.

We start by determining the limiting distribution of the number of $\ell$-APs to be either a Poisson distribution or a Gaussian distribution. Let  $\sigma_{\ell}:=\sqrt{\VV(\RVnumberAP{\ell}{})}$ denote the standard deviation of $\RVnumberAP{\ell}{}$.
\begin{theorem}[Univariate limiting distributions]\label{thm:mainUnivariate}
Let $\ell\ge 3$ be either a constant, or $\ell=\ell(n)\to+\infty$ satisfying $\ell/\log n\to 0$, and let $0< p=p(n)=o(1)$.  
\begin{enumerate}[(a)]
\item If $n^2p^{\ell}/(\ell-1)\to c$, for some $c\in\RR_+$, then  \label{thm:mainUnivariatePoisson}
$
\RVnumberAP{\ell}{}  \stackrel{d}{\longrightarrow} \Po\left(c/2\right).
$
\item If $n^2p^{\ell}/(\ell-1)\to+\infty$, then \label{thm:mainUnivariateNormal}
$
\left(\RVnumberAP{\ell}{}-\EE(\RVnumberAP{\ell}{})\right)\sigma_\ell^{-1} \stackrel{d}{\longrightarrow}\No(0,1).
$
\end{enumerate}
\end{theorem}
While a priori $\ell$ could be as large as $n$, it is easy to see that the random subset $[n]_p$ with $p=o(1)$ (i.e.\ in the sparse regime) asymptotically almost surely (a.a.s.) does not contain any $\ell$-APs with $\ell=\ell(n)\ge C\log n$ for any constant $C>0$. This follows by a first moment argument, since 
\begin{equation} \label{eq:rem}
\EE(\RVnumberAP{\ell}{})\stackrel{Cl.~\ref{obs:numberAP}}{=} (1\pm o(1))\frac{n^2p^{\ell}}{2(\ell-1)}\le\exp\left[ 2\log n -  C \log n \log (p^{-1}) \right] =o(1),
\end{equation}
and thus by Markov's inequality $\PP(\RVnumberAP{\ell}{}=0)\to 1$. In other words, Theorem~\ref{thm:mainUnivariate} is optimal concerning the range of $\ell$.\footnote{Except for cases where we can only expect convergence along subsequence, for instance if $\ell=\ell(n)$  alternates (periodically) between two or more constants.}

We remark that for constant $\ell\ge 3$, Theorem~\ref{thm:mainUnivariate} hardly comes as a surprise since $\RVnumberAP{\ell}{}$ is a sum of ``weakly dependent'' Bernoulli random variables. The Gaussian approximation follows then from a sufficient criterion due to Mikhailov (cf.\ Theorem~\ref{Thm:MikhailovNormality}), while the Chen-Stein method (cf.\ Theorem~\ref{thm:AGG}) yields the Poisson approximation. Yet, we could not find a proof of this result in the literature. The fact that the proof carries through for growing $\ell=\ell(n)\to+\infty$ is largely due to the fact that the expectation in~\eqref{eq:rem} decreases exponentially quickly in $\ell$.

 Our main result characterises the bivariate fluctuations of the pair $(\RVnumberAP{\ell_1}{},\RVnumberAP{\ell_2}{})$ when both random variables are within their respective Gaussian regimes, as determined in Theorem~\ref{thm:mainUnivariate}. 

\begin{theorem}[Bivariate fluctuations for APs of different lengths]\label{thm:mainBivariate}
   For $i\in\{1,2\}$, let $\ell_i\ge 3$ be either a constant, or $\ell_i=\ell_i(n)\to+\infty$, such that we have $\ell_2<\ell_1$ (point-wise) and $\ell_1/\log n\to 0$. Let $0< p=p(n)< 1$ be such that $p\ell_1^{9}\to 0$ and $n^2p^{\ell_1}\ell_1^{-9}\to+\infty$. Then we have 
   $$
   \left(\frac{\RVnumberAP{\ell_1}{}-\EE(\RVnumberAP{\ell_1}{})}{\sigma_{\ell_1}},\frac{\RVnumberAP{\ell_2}{}-\EE(\RVnumberAP{\ell_2}{})}{\sigma_{\ell_2} }\right)\stackrel{d}{\longrightarrow} \Gaussian{0}{0}{1}{\covarianceLeadingConstant{\ell_1}{\ell_2}}{1}, 
   $$
where $\covarianceLeadingConstant{\ell_1}{\ell_2}$ satisfies 
$$
\begin{cases}
\covarianceLeadingConstant{\ell_1}{\ell_2}=0, & \text{ if } np^{\ell_1-1}\ell_1\to 0;\\
0<\covarianceLeadingConstant{\ell_1}{\ell_2}< 1, & \text{ if } np^{\ell_1-1}\ell_1\to c\in\RR_+ \vee\left[np^{\ell_1-1}\ell_1\to+\infty \wedge\ell_2\text{ is a constant }\right];\\
\covarianceLeadingConstant{\ell_1}{\ell_2}=1, &\text{ if } np^{\ell_1-1}\ell_1\to+\infty \wedge \ell_2=\ell_2(n)\to+\infty.
\end{cases}
$$
\end{theorem}
Interestingly, the strength of the correlation is characterised by the asymptotic behaviour of the function 
\begin{equation}\label{eq:psi}
\psi_{\ell_1}=\psi_{\ell_1}(n):=np^{\ell_1-1}\ell_1,
\end{equation}
which originates from the combinatorial structure of tuples of overlapping APs. There are two structures, \emph{loose pairs} and \emph{overlap pairs} (see Definition~\ref{def:LooseOverlapPairs}), which compete to dominate the centralised second moments of the pair $(\RVnumberAP{\ell_1}{},\RVnumberAP{\ell_2}{})$. The function $\psi_{\ell_1}$ is obtained as the ratio of the contribution of loose pairs by that of overlap pairs (of $\ell_1$-APs); when $ \psi_{\ell_1}\to 0$, overlap pairs dominate, and when $\psi_{\ell_1}\to+\infty$, loose pairs dominate. We call the former the \emph{overlap pair regime}, and the latter the \emph{loose pair regime}. An explicit expression of $\covarianceLeadingConstant{\ell_1}{\ell_2}$ is given in Lemma~\ref{lem:kappa} and its proof; its derivation is surprisingly intricate and involves an integral representation.

Furthermore, we want to highlight that when $\ell_2=\ell_2(n)\to+\infty$ (and thus also $\ell_1=\ell_1(n)\to+\infty$), the  random variables $\RVnumberAP{\ell_1}{}$ and $\RVnumberAP{\ell_2}{}$ are either asymptotically uncorrelated, or converge to the \emph{same} random variable (once renormalised). However, in all other cases, there exists a regime where the asymptotic correlation is non-trivial.

Lastly, we remark that the conditions are slightly more restrictive due to technical reasons, we strongly believe that the result remains true under the weaker assumptions $n^2p^{\ell_1-1}\ell^{-1}\to+\infty$ and $p\to 0$, which characterise the sparse Gaussian regime for $\ell_1$-APs, cf.\ Theorem~\ref{thm:mainUnivariate}\eqref{thm:mainUnivariateNormal}.

\subsection{Related work}
In the literature, the study of $\RVnumberAP{\ell}{}$ for random subsets of the integers is largely focused on $\ell\ge3$ being a constant and estimating the probability of large deviations from its mean, i.e.\ the \emph{upper tail} probabilities $\PP\left(\RVnumberAP{\ell}{}\ge (1+\eps)\EE(\RVnumberAP{\ell}{})\right)$, and the \emph{lower tail} probabilities $\PP\left(\RVnumberAP{\ell}{}\le (1-\eps)\EE(\RVnumberAP{\ell}{})\right)$. For a recent survey on large deviations in random graphs (and related combinatorial structures) see~\cite{Chatterjee}. 

For the upper tail, Janson and Ruci\'nski~\cite{JansonRucinski} obtained upper and lower bounds on $-\log\PP\left(\RVnumberAP{\ell}{}\ge (1+\eps)\EE(\RVnumberAP{\ell}{})\right)$ being apart by a factor of $\log(1/p)$ by extending an earlier result by Janson, Oleszkiewicz, and Ruci\'nski~\cite{JansonOleszkiewiczRucinski} on large deviations for subgraph counts in random graphs. Subsequently, Warnke~\cite{Warnke} closed this gap by proving that
\begin{align*}
-\log \PP\!\left( \RVnumberAP{\ell}{} \geq (1 + \varepsilon) \EE( \RVnumberAP{\ell}{} ) \right) = \Theta_\varepsilon\!\left( \Phi( \EE( \RVnumberAP{\ell}{} ) ) \right) , \qquad \Phi(x) := \min\{ x, \sqrt{x} \log(1/p) \},
\end{align*}
and also supplying the dependency on $\eps$ of the implied constants in $\Theta_\eps$. Notably, provided that $p$ is in the loose pair regime (more precisely, $\psi_{\ell}\ge \log n$, where $\psi_{\ell}=np^{\ell-1}\ell$ as in~\eqref{eq:psi}) the results in~\cite{Warnke} also extend to \emph{moderate variations}, i.e.\ events of the form $\{\RVnumberAP{\ell}{}\ge \EE(\RVnumberAP{\ell}{})+t\}$ for any $t\ge \sigma_{\ell}$. Complementing these results, Bhattacharya, Ganguly, Shao, and Zhao~\cite{BhattacharyaEtAl} pinned down the precise large deviation rate function for ``sufficiently large'' $p$. By contrast to the approach in~\cite{Warnke}, the proof in~\cite{BhattacharyaEtAl} builds on the non-linear large deviation principle by Chatterjee and Dembo~\cite{ChatterjeeDembo} and its refinement due to Eldan~\cite{Eldan} in terms of the concept of \emph{Gaussian width}, a particular notion of complexity. Recently, Bri\"et and Gopi~\cite{BrietGopi} derived an upper bound on the Gaussian width leading to an improvement of the lower bound on $p$ given in~\cite{BhattacharyaEtAl}. The special case $\ell=3$ was already included in~\cite{ChatterjeeDembo}.

On the other hand, the lower tail has received less attention: for all constants $\ell\ge 3$, Janson and Warnke~\cite{JansonWarnke} determined the large deviation rate function up to constants to be
\begin{align*}
-\log \PP\!\left( \RVnumberAP{\ell}{} \geq (1 - \varepsilon) \EE( \RVnumberAP{\ell}{} ) \right) = \Theta(\eps^2 \min\{\EE(\RVnumberAP{\ell}{}),np\}),
\end{align*}
while Mousset, Noever, Panagiotou, and Samotij~\cite{MoussetNoeverPanagiotouSamotij} concentrated on the probability of $[n]_p$ to be $\ell$-AP free, and expressed $-\log \PP(\RVnumberAP{\ell}{}=0)$ as an alternating sum of certain \emph{joint cumulants} defined in terms of the \emph{dependency graph} associated to  $\RVnumberAP{\ell}{}$. The results on $\ell$-APs in~\cite{MoussetNoeverPanagiotouSamotij} hold only for $p$ within the overlap pair regime ($\psi_{\ell}=o(1)$, where $\psi_{\ell}=np^{\ell-1}\ell$ as in~\eqref{eq:psi}). 

We complement the literature on large and moderate deviations by considering typical deviations and thereby determining the limiting distribution of $\RVnumberAP{\ell}{}$ not only for all constants $\ell\ge 3$  but also when $\ell=\ell(n)\to+\infty$. Additionally, we also investigate the interaction of the number of APs of different length occurring in $[n]_p$, i.e.\ typical fluctuations of the pair $(\RVnumberAP{\ell}{},\RVnumberAP{\ell'}{})$. Strikingly, we find a significantly different behaviour of their bivariate fluctuations in the overlap pair regime, as compared to the loose pair regime. By contrast to the results on moderate deviations in~\cite{Warnke} or the result in~\cite{MoussetNoeverPanagiotouSamotij} which work only in one of the two regimes, we employ the same approach in both regimes.

\subsection{Proof method and outline}
The main goal of this article lies in the analysis of bivariate fluctuations of the pair $(\RVnumberAP{\ell_1}{},\RVnumberAP{\ell_2}{})$ based on the method of moments: we show that the joint moments of $(\RVnumberAP{\ell_1}{},\RVnumberAP{\ell_2}{})$, once centred and rescaled, converge to the moments of a Gaussian random vector, which ensures the convergence in distribution. More formally, we apply the combination of the following two classical results.
\begin{theorem}[e.g.\ Theorem~30.2 in~\cite{Billingsley08}]\label{thm:MethodOfMoments}
Let $\cY$ be a random variable which is determined by its moments, and let $(Y_n)_{n\in\NN}$ be a sequence of random variables having finite moments of all orders. If $\lim_{n\to+\infty}\EE(Y_n^k)=\EE(\cY^k)$ for all $k\in\NN$, then $Y_n\stackrel{d}{\tendsto{n}{+\infty}}\cY$.
\end{theorem}
The same principle transfers to multivariate random variables, by application of the Cram\'er-Wold device.
\begin{theorem}[Cram\'er-Wold device, e.g.\ Theorem~29.4 in~\cite{Billingsley08}]\label{thm:CramerWoldDevice}
For any $r\in \NN$, let $\cY=(\cY_1,\dots, \cY_r)$ and $Y_n=(Y_{n,1},\dots, Y_{n,r})$, $n\in\NN$, be random vectors. Then $Y_n\stackrel{d}{\longrightarrow}\cY$ if and only if 
$$
\sum_{i = 1}^{r}u_iY_{n,i}\stackrel{d}{\longrightarrow}\sum_{i = 1}^{r}u_i\cY_i, \quad \forall u_1,\dots, u_k\in\RR.
$$
\end{theorem}
Our approach for the analysis of the (normalised) joint moments was inspired by a recent result of Gao and Sato~\cite{GaoSato}  determining the limiting distribution of the number of matchings of size $\ell=\ell(n)$ in $G(n,p)$ to be either a Normal or a Log-normal distribution. It is well-known that the odd moments of a centred, multivariate Gaussian distribution vanish, while the even moments can be expressed combinatorially: for $k\in\NN$ the $2k$-th moment is given by a sum over all perfect matchings of the set $[2k]$. Thus the heart of our proof lies in showing that the (even and centred) joint moments  of $(\RVnumberAP{\ell_1}{},\RVnumberAP{\ell_2}{})$ are dominated by a similar matching structure. In fact, we will see that this combinatorial structure is encoded in the \emph{dependency graph} $\Gamma$ (cf.\ Definition~\ref{def:depGraph}) associated with the pair $(\RVnumberAP{\ell_1}{},\RVnumberAP{\ell_2}{})$. Depending on the range of $p$, the main contribution will come from matchings consisting of overlap pairs and/or loose pairs, and can be determined explicitly. It then remains to bound the contributions of all non-matching configurations. This last step is based on an algorithmic exploration of the components in $\Gamma$; a similar argument was previously used by Bollob\'as, Cooley, Kang, and the second author~\cite{BollobasCooleyKangKoch} in the context of jigsaw percolation on random hypergraphs. By contrast, in~\cite{GaoSato} this last step was based on the \emph{switching method} introduced by McKay~\cite{McKay}, which turned out to be difficult to apply in the setting of APs due to their arithmetic structure.

We close with an outline of the article: Section~\ref{Sec:Preliminaries} focusses on counting APs and pairs of APs, and deriving the joint second moments from these. Since we require a high level of precision, the counting argument for loose pairs of APs turns out to be surprisingly challenging. In Section~\ref{Sec:Univariate} we complete the proof of Theorem~\ref{thm:mainUnivariate} based on two sufficient criteria from the literature. The higher joint moments of the pair $(\RVnumberAP{\ell_1}{},\RVnumberAP{\ell_2}{})$ are analysed in Section~\ref{Sec:GaussianRegimeBivariate}, where we also complete the proof of Theorem~\ref{thm:mainBivariate} and provide an alternative proof of Theorem~\ref{thm:mainUnivariate}\eqref{thm:mainUnivariateNormal}. We then conclude with a discussion of open problems in Section~\ref{Sec:Concluding}.

%%%%%%%%%%%%%%%%%%%%%%%%%%%%%%%%%%%%%%%%%%%%%%%%%%%%% $
%%%%%%%%%%%%%%%%%%%%%%%%%%%%%%%%%%%%%%%%%%%%%%%%%%%%% $
\section{Preliminaries: counting APs and pairs of APs}\label{Sec:Preliminaries}
We start out with determining the asymptotics related to the set of APs in $[n]$. First, we consider the total number of $\ell$-APs, denoted by $\numberAP{\ell}$.
\begin{claim}\label{obs:numberAP}
For any $3\le \ell=\ell(n)\le n$, we have 
$$
\numberAP{\ell}= 
\begin{cases}
(1\pm o(1))\frac{n^2}{2(\ell-1)}& \text{ if } \ell/n\to 0,\\
\Theta(n)& \text{ if } \ell/n\to c\in (0,1),\\
(1\pm o(1))(n-\ell+1)& \text{ if } \ell/n\to 1. 
\end{cases}
$$
In particular, the following asymptotics holds for all $3\le \ell=\ell(n)\le n$:
$$
\numberAP{\ell}= \Theta(n(n-\ell+1)\ell^{-1}).
$$
Furthermore, for any $3\le \ell=\ell(n)=o(n)$, we have 
$$
\EE(\RVnumberAP{\ell}{})=\numberAP{\ell}p^\ell=(1\pm o(1))\frac{n^2p^\ell}{2(\ell-1)}.
$$
\end{claim}

\begin{proof}
Let $R:=\left(\frac{n-1}{\ell-1}-\left\lfloor\frac{n-1}{\ell-1}\right\rfloor\right)\cdot (\ell-1)$ and observe that $0\le R\le \ell-2$. We have
\begin{align*}
\numberAP{\ell}&=\sum_{\comDiff{}=1}^{\left\lfloor\frac{n-1}{\ell-1}\right\rfloor}\sum_{m=1}^{n}\Indens{m+(\ell-1)\comDiff{}\le n}=\sum_{\comDiff{}=1}^{\left\lfloor\frac{n-1}{\ell-1}\right\rfloor}(n-\comDiff{}(\ell-1))\\
&=\left\lfloor\frac{n-1}{\ell-1}\right\rfloor\cdot n- (\ell-1)\binom{\left\lfloor\frac{n-1}{\ell-1}\right\rfloor+1}{2}=\frac{n(n-\ell+1)}{2(\ell-1)}+f(R,\ell),
\end{align*}
where $f(R,\ell):= \frac{(R+1)(\ell-1)-(R+1)^2}{2(\ell-1)}$. Furthermore, we observe that for all $\ell$ we have $0\le f(R,\ell)\le (\ell-1)/8$. It remains to distinguish three cases:

$\bullet $ if $\ell/n\to 0$, then $f(R,\ell)=o(n)=o(n^2/\ell)$ and the claim follows immediately,

$\bullet $ if $\ell/n\to c$ for some constant $c\in (0,1)$, then $f(R,\ell)=O(n)$ and again the claim follows immediately,

$\bullet $ if $ \ell/n \to 1$, the $ \ell $-AP contained in $ [n] $ is clearly an interval, hence the number of such choices is $ n - \ell + 1 $, completing the proof.
\end{proof}

\subsection{Loose pairs and overlap pairs}  

Next, we consider pairs of APs of potentially different lengths, and distinguish them by the size of their intersection.
\begin{definition}\label{def:LooseOverlapPairs} Let $3\le \ellTwo=\ellTwo(n)\le\ellOne=\ellOne(n)\le n$. 
\begin{enumerate}[(a)]
\item For any $r\in [\ellTwo]$, we define 
$$
\setMixedPartitionPairs{\ellOne}{\ellTwo}{r}:=\left\{(T,T')\in\setAP{\ellOne}{}\times \setAP{\ellTwo}{}\colon |T\cap T'|= r \right\}
$$
to be the set of (ordered) pairs of APs intersecting in precisely $r$ elements. 
\item We say that a pair $(T,T')\in \setAP{\ellOne}{}\times \setAP{\ellTwo}{}$ is  a \emph{loose pair} if $|T\cap T'|=1$. We use the shorthand $\setMixedLoosePairs{\ellOne}{\ellTwo}:=\setMixedPartitionPairs{\ellOne}{\ellTwo}{1}$ for the set of all loose pairs.
\item We say that a pair $(T,T')\in \setAP{\ellOne}{}\times \setAP{\ellTwo}{}$ is  an \emph{overlap pair} if $|T\cap T'|=\ellTwo$, or equivalently $T'\subseteq T$. We use the shorthand $\setMixedOverlapPairs{\ellOne}{\ellTwo}:=\setMixedPartitionPairs{\ellOne}{\ellTwo}{\ellTwo}$ for the set of all overlap pairs.
\item We denote the cardinalities of these sets by $ \numberMixedPartitionPairs{\ellOne}{\ellTwo}{\cdot} := |\setMixedPartitionPairs{\ellOne}{\ellTwo}{\cdot}| $, $ \numberMixedLoosePairs{\ellOne}{\ellTwo} := |\setMixedLoosePairs{\ellOne}{\ellTwo}| $, and $ \numberMixedOverlapPairs{\ellOne}{\ellTwo} := |\setMixedOverlapPairs{\ellOne}{\ellTwo}| $, respectively. Furthermore, whenever $\ellOne=\ellTwo$ we drop one of the lower indices, e.g.\ we use $\numberPartitionPairs{\ellOne}{2} := \numberMixedPartitionPairs{\ellOne}{\ellOne}{2}$.
\end{enumerate}
\end{definition}

Computing the asymptotic behaviour of the number of overlap pairs is a Corollary of Claim~\ref{obs:numberAP}.
\begin{corollary}\label{obs:numberOverlapPair}
For all $3\le \ellTwo=\ellTwo(n)\le \ellOne=\ellOne(n)=o(n)$ we have
$$
\numberMixedOverlapPairs{\ellOne}{\ellTwo} =\Theta(1)\cdot n^2(\ellOne-\ellTwo+1)/\ellTwo.
$$
\end{corollary}
\begin{proof}
Note that the number of overlap pairs $ (T_1, T_2)\in \setMixedOverlapPairs{\ellOne}{\ellTwo} $ is equal to
$\numberAP{\ellOne}\cdot M$, where $M$ is the number of $\ellTwo$-APs in $[\ellOne]$. Indeed, by Claim~\ref{obs:numberAP}, we have $M=\Theta(\ellOne(\ellOne-\ellTwo+1)/\ellTwo)$ and $\numberAP{\ellOne}=\Theta(n^2\ellOne^{-1})$
and the statement follows.
\end{proof}

Similarly, we obtain an upper bound on the number of pairs intersecting in precisely $r$ elements for $2\le r\le \ellTwo-1$. Despite being somewhat crude, this bound will suffice for our purposes.
\begin{claim}\label{claim:numberOtherPairs}
For any $3\le \ellTwo=\ellTwo(n)\le \ellOne=\ellOne(n)= o(n)$ and $2\le r\le \ellTwo-1$ we have 
$$
\numberMixedPartitionPairs{\ellOne}{\ellTwo}{r}=O(n^2\ellOne(\ellTwo)^2).
$$
Furthermore, in case $r\ge \lfloor 2\ellTwo/3\rfloor+1$, we  have 
$$
\numberMixedPartitionPairs{\ellOne}{\ellTwo}{r}=O(n^2(\ellOne-r+1)(\ellTwo-r+1)/\ellTwo).
$$
\end{claim}
\begin{proof}
Note that a pair $(T,T')\in\setPartitionPairs{\ellOne,\ellTwo}{r}$ is already uniquely determined by choosing the first AP $T$, for which there are at most $O(n^2\ellOne^{-1})$ many choices by Claim~\ref{obs:numberAP}; and then fixing the relative position of the first two intersection elements within $T$ and $T'$, for which there are at most $\ellOne^2$ and $(\ellTwo)^2$ many choices, respectively. The first claim follows by multiplying.

As for the second bound, assume that $r\ge 2\ellTwo/3$, then any pair $(T,T')\in\setPartitionPairs{\ellOne,\ellTwo}{r}$ induces an overlap pair consisting of the $\ellOne$-AP $T$ and the $r$-AP $T\cap T'$. By definition the number of such pairs is $\numberMixedOverlapPairs{\ellOne}{r}$ and thus at most $O(n^2(\ellOne - r +1)/\ellTwo)$, by Corollary~\ref{obs:numberOverlapPair}. Next, observe that once $T$ and $T\cap T'$ are chosen, the common difference of $T'$ needs to be a divisor of the common difference of $T\cap T'$. However, since $r\ge \lfloor 2\ellTwo/3\rfloor+1$ we have $|T'\setminus T|\le \ellTwo-\lfloor 2\ellTwo/3\rfloor-1\le\ellTwo/3<r-1$, implying that both $T\cap T'$ and $T'$ have the same common difference. So we may only choose how many elements of $T'\setminus T$ are smaller than the smallest element of $T\cap T'$, the number of choices is at most $\ellTwo-r+1$. Hence in total we obtain the claimed upper bound.
\end{proof}

By contrast, determining the asymptotics of the number of loose pairs is much more difficult. In the following we will use the convention that $1/0=+\infty$, $\min\{x,+\infty\}=x$, and $\overline{x}:=1-x$ for all $x\in[0,1]$. Moreover, for any $3\le \ell=\ell(n)\le n$ we define a function $\measurePos{\ell}$ by setting
\begin{align}\label{def:measurePos}
\measurePos{\ell}(x) := \frac{1}{\ell-1} \sum_{\pos{} = 1 }^{\ell} \Indens{x\geq(\pos{}-1)/(\ell-1)},
\end{align}
for all $x\in[0,1]$. Furthermore, we define functions $\covarianceFunction{\ell}\colon [0,1]\mapsto[0,1]$ by the following Lebesgue-Stieltjes integral
\begin{align}\label{def:covarianceFunction}
\covarianceFunction{\ell}(x):=\int_0^1 \min\left\{\frac{x}{\overline{\posRescaled{}}},\frac{\overline{x}}{\posRescaled{}}\right\}d\measurePos{\ell}(\posRescaled{}).
\end{align} 
We start by proving two technical properties of these functions
\begin{claim}\label{claim:covarianceFunctionConstant}
For any constant $\ell\ge 3$ the function $\covarianceFunction{\ell}$ is non-negative and has the following properties:
\begin{enumerate}[(a)]
\item Uniformly for all $1/3\le x\le 2/3$, we have
\begin{equation}
\covarianceFunction{\ell}(x)\ge \frac{1}{2(\ell-1)}.
\end{equation} 
\item For all $0\le x\le \frac{1}{2(\ell-1)}$ we have
\begin{equation}
\covarianceFunction{\ell}(x)= \frac{1}{\ell-1}+x\harm{\ell-2},
\end{equation} 
where $\harm{t}:=\sum_{j=1}^{t}1/j$ denotes the $t$-th harmonic number.
\end{enumerate}
\end{claim}
\begin{proof}
For the first claim, we note that $\min\left\{\frac{x}{\overline{\posRescaled{}}},\frac{\overline{x}}{\posRescaled{}}\right\} \ge 1/2$ for all $1/3\le \posRescaled{}\le 2/3$ and $1/3\le x\le 2/3 $. We conclude by noting that there is at least one $\pos{}$ in $\{1,2,\ldots, \ell\}$ such that $1/3\le(\pos{}-1)/(\ell-1)\le 2/3$.

For the second claim, let $x\le\frac{1}{2(\ell-1)} $ and note that for all $1\le \pos{}\le \ell-1$ we have $1-\frac{\pos{}-1}{\ell-1}\ge \frac{1}{\ell-1}>x$ implying that 
$$
\min\left\{\frac{x}{1-\frac{\pos{}-1}{\ell-1}},\frac{1-x}{\frac{\pos{}-1}{\ell-1}}\right\}=\frac{x(\ell-1)}{\ell-\pos{}}.
$$
Therefore, we obtain
$$
\covarianceFunction{\ell}(x)=\frac{x+(1-x)}{\ell-1}+x\sum_{\pos{}=2}^{\ell-1}\frac{1}{\ell-\pos{}}=\frac{1}{\ell-1} +x\harm{\ell-2},
$$
as claimed.
\end{proof}

Next, let the entropy function $\covarianceFunction{\infty}\colon [0,1]\mapsto[0,1]$ be defined by 
\begin{align}\label{def:covarianceFunctionInfty}
\covarianceFunction{\infty}(x):=
\begin{cases}
x\log(1/x)+\overline{x}\log(1/\overline{x})& \text{ if } 0<x<1,\\
0&\text{ if } x=0\vee x=1,
\end{cases}
\end{align} 
and observe that $\covarianceFunction{\infty}$ is continuous on $[0,1]$. The next statement shows that $\covarianceFunction{\infty}$ is obtained naturally from $\covarianceFunction{\ell}$ when $\ell=\ell(n)\to+\infty$.
\begin{claim}\label{claim:covarianceFunction}
For any $\ell=\ell(n)\to+\infty$ with $\ell=o(n)$, the function $\covarianceFunction{\ell}$ converges to $\covarianceFunction{\infty}$ in $L^2$ as $n\to+\infty$.
\end{claim}
\begin{proof}
We first observe that $\{d\measurePos{\ellOne}\}_{\ellOne \in \mathbb{N}}$ converges weakly to the uniform measure on $[0,1]$ as $n\to+\infty$. Furthermore, the function $\posRescaledOne\mapsto \min\left\{\frac{x}{\overline{\posRescaledOne}},\frac{\overline{x}}{\posRescaledOne}\right\}$ is bounded and continuous for all $x\in[0,1]$, and thus we have 
\begin{align*}
\covarianceFunction{\ell}(x)=\int_0^1 \min\left\{\frac{x}{\overline{\posRescaledOne}},\frac{\overline{x}}{\posRescaledOne}\right\}d\measurePos{\ellOne}(\posRescaledOne)&=(1\pm o(1))\int_0^1 \min\left\{\frac{x}{\overline{\posRescaledOne}},\frac{\overline{x}}{\posRescaledOne}\right\}d\posRescaledOne.
\end{align*}
Moreover, for all $x\in(0,1)$ we have
\begin{align*}
\int_0^1 \min\left\{\frac{x}{\overline{\posRescaledOne}},\frac{\overline{x}}{\posRescaledOne}\right\}d\posRescaledOne& =\int_0^1  \left( \frac{x}{\overline{\posRescaledOne} } \Indens{ x \leq \overline{\posRescaledOne} } + \frac{\overline{x}}{\posRescaledOne}  \Indens{ x \geq \overline{\posRescaledOne} } \right) d\posRescaledOne \\
                & = x \int_x^1 \frac{d \overline{\posRescaledOne}}{\overline{\posRescaledOne} } + \overline{x} \int_{\overline{x}}^1 \frac{d\posRescaledOne}{\posRescaledOne} \\
                & = x \log\left( 1/x \right) + \overline{x} \log\left(1/\overline{x} \right),
\end{align*}
and this expression extends continuously for $x\in[0,1]$. In other words, $\covarianceFunction{\ell}$ converges point-wise to $\covarianceFunction{\infty}$. 

However, since uniformly for all $x\in[0,1]$ we have $\covarianceFunction{\ell}(x)^2\le1$, the Dominated Convergence Theorem implies that also $\covarianceFunction{\ell}\to\covarianceFunction{\infty}$ in $L^2$.
\end{proof}

With this preparation we will now determine the number of loose pairs  asymptotically.
\begin{lemma}\label{lem:mix-loose-pair}
 Let $3\le \ellTwo=\ellTwo(n)\le \ellOne=\ellOne(n)= o(n)$.
\begin{enumerate}[(a)]
\item If both $\ellOne$ and $\ellTwo$ are constant, then we have
\begin{align*}
\frac{\numberMixedLoosePairs{\ellOne}{\ellTwo} }{n^3} \tendsto{n}{+\infty}  \int_0^1\covarianceFunction{\ellOne}(t)\covarianceFunction{\ellTwo}(t)dt>0. 
\end{align*}
\item If $\ellOne=\ellOne(n)\to+\infty$, but $\ellTwo$ is a constant, then we have 
\begin{align*}
\frac{\numberMixedLoosePairs{\ellOne}{\ellTwo} }{n^3} \tendsto{n}{+\infty} \int_0^1\covarianceFunction{\infty}(t)\covarianceFunction{\ellTwo}(t)dt>0.
\end{align*}
\item If  $\ellTwo=\ellTwo(n)\to+\infty$, then we obtain
\begin{align*}
\frac{\numberMixedLoosePairs{\ellOne}{\ellTwo} }{n^3} \tendsto{n}{+\infty} \int_0^1\covarianceFunction{\infty}(t)^2dt=\frac{5}{6}-\frac{\pi ^2}{18}=0.2850\ldots \,.
\end{align*}
\end{enumerate}
\end{lemma}
\begin{proof}
Let $\comDiffMaxOne:=\left\lfloor\frac{n-1}{\ellOne-1}\right\rfloor$ and $\comDiffMaxTwo:=\left\lfloor\frac{n-1}{\ellTwo-1}\right\rfloor$. We enumerate the elements $(T,T')\in\numberAP{\ellOne}\times \numberAP{\ellTwo}$, with $T=(T(1),\dots,T(\ellOne))$ and $T'=(T'(1),\dots,T'(\ellTwo))$, by fixing the common differences $ (\comDiffOne,\comDiffTwo) \in [\comDiffMaxOne]\times [\comDiffMaxTwo] $, and the unique intersection point $\interPoint\in[n]$ together with its positions $ (\posOne,\posTwo) \in [\ellOne]\times[\ellTwo] $ within $(T,T')$. Then both $\ellOne$-APs are to be contained in $[n]$ if and only if 
\begin{align*}
1\le T(1) \:\:\wedge\:\: 1\le T'(1) \:\:\wedge\:\: T(\ellOne)\le n \:\:\wedge\:\: T'(\ellTwo)\le n.
\end{align*}
Expressing $T(1)$, $T'(1)$, $T(\ellOne)$, and $T'(\ellTwo)$ in terms of $\interPoint$, $\posOne$, $\posTwo$, $\comDiffOne$, and $\comDiffTwo$, this is equivalent to
\begin{align*}
1 + \max\left\{(\posOne-1)\comDiffOne,(\posTwo-1)\comDiffTwo\right\} \leq \interPoint \leq n - \max\left\{(\ellOne - \posOne)\comDiffOne,(\ellTwo - \posTwo)\comDiffTwo\right\}.
\end{align*}

In other words, the number of valid choices for $\interPoint$ is 
\begin{align*}
\left(n-\max\left\{(\posOne-1)\comDiffOne,(\posTwo-1)\comDiffTwo\right\}-\max\left\{(\ellOne-\posOne)\comDiffOne,(\ellTwo-\posTwo)\comDiffTwo\right\}\right)_+,
\end{align*}
with $ x_+ := \max\{x, 0\} = x\Indens{x \geq 0} $, and by summing over all choices for $(\posOne, \posTwo, \comDiffOne, \comDiffTwo) \in [\ellOne]\times[\ellTwo]\times[\comDiffMaxOne]\times[\comDiffMaxTwo] $, we obtain
\begin{align*}
\numberMixedLoosePairs{\ellOne}{\ellTwo} & = \!\!  \sum_{  (\posOne, \posTwo, \comDiffOne, \comDiffTwo)  }  \!\! ( n-\max\left\{(\posOne-1)\comDiffOne,(\posTwo-1)\comDiffTwo\right\}-\max\left\{(\ellOne-\posOne)\comDiffOne,(\ellTwo-\posTwo)\comDiffTwo\right\} )_+ \, .
\end{align*}
It turns out to be convenient to divide this quantity by $n$ to obtain
\begin{align}\label{eq:loosePairsFormula}
\frac{\numberMixedLoosePairs{\ellOne}{\ellTwo}}{n}&=\sum_{  (\posOne, \posTwo, \comDiffOne, \comDiffTwo)  }f\left(\frac{\posOne-1}{\ellOne-1},\frac{\posTwo-1}{\ellTwo-1},\frac{(\ellOne-1)\comDiffOne}{n},\frac{(\ellTwo-1)\comDiffTwo}{n}\right)
\end{align}
where the function $f\colon [0,1]^4\to [0,1]$ is defined by
$$
f(\posRescaledOne,\posRescaledTwo,\comDiffRescaledOne,\comDiffRescaledTwo):=\left(1-\max\{\posRescaledOne\comDiffRescaledOne,\posRescaledTwo\comDiffRescaledTwo\}-\max\{(1-\posRescaledOne)\comDiffRescaledOne,(1-\posRescaledTwo)\comDiffRescaledTwo\}\right)_+ \, .
$$
Now note that we have
\begin{align*}
\comDiffMaxOne=(1\pm O(\ellOne/n))\frac{n}{\ellOne-1}\quad \text{ and }\quad \comDiffMaxTwo=(1\pm O(\ellTwo/n))\frac{n}{\ellTwo-1},
\end{align*}
implying
\begin{align*}
\frac{(\ellOne-1)\comDiffOne}{n}=(1\pm O(\ellOne/n))\frac{\comDiffOne}{\comDiffMaxOne}\quad \text{ and }\quad
\frac{(\ellTwo-1)\comDiffTwo}{n}=(1\pm O(\ellTwo/n))\frac{\comDiffTwo}{\comDiffMaxTwo},
\end{align*}
and thus it is not hard to show that there exists a constant $C>0$ such that for all $1\le \posOne\le \ellOne$ and $1\le \posTwo\le \ellTwo$ we have 
$$
\left|f\left(\frac{\posOne-1}{\ellOne-1},\frac{\posTwo-1}{\ellTwo-1},\frac{(\ellOne-1)\comDiffOne}{n},\frac{(\ellTwo-1)\comDiffTwo}{n}\right)-f\left(\frac{\posOne-1}{\ellOne-1},\frac{\posTwo-1}{\ellTwo-1},\frac{\comDiffOne}{\comDiffMaxOne},\frac{\comDiffTwo}{\comDiffMaxTwo}\right)\right|\le C\cdot\frac{\ellOne}{n}.
$$
Furthermore, let 
$$
\measureComDiff{n}(x,x'):=\frac{1}{\comDiffMaxOne\comDiffMaxTwo}\sum_{(\comDiffOne,\comDiffTwo)}\Indens{x\leq\comDiffOne}\Indens{x'\leq\comDiffTwo}
$$
and observe that $\{d\measureComDiff{n}\}_{n\in\NN}$ converges weakly to the uniform measure on $[0,1]^2$. Since $f$ is bounded and continuous, we therefore have
\begin{align}\label{eq:loosePairsFormula2}
\frac{\numberMixedLoosePairs{\ellOne}{\ellTwo}}{n\comDiffMaxOne\comDiffMaxTwo}=(1\pm o(1))\sum_{(\posOne,\posTwo)}\int_{ [0, 1]^2 } f\left(\frac{\posOne-1}{\ellOne-1},\frac{\posTwo-1}{\ellTwo-1},\comDiffRescaledOne,\comDiffRescaledTwo\right)d\comDiffRescaledOne d\comDiffRescaledTwo.
\end{align}

 The next goal is to deal with the positive part of the function $f$: we note that 
\begin{align*}
(R - Q)_+ =  R - \min\{R, Q\}
\end{align*}
and so,  for any $(\posRescaledOne,\posRescaledTwo,\comDiffRescaledOne,\comDiffRescaledTwo)\in[0,1]^4$, by setting 
\begin{align*}
R:=&\min\{1-\posRescaledOne\comDiffRescaledOne,1-\posRescaledTwo\comDiffRescaledTwo\}, \\
Q:=&\max\{(1-\posRescaledOne)\comDiffRescaledOne,(1-\posRescaledTwo)\comDiffRescaledTwo\},
\end{align*}
 we obtain
\begin{align*}
 f\left(\posRescaledOne,\posRescaledTwo,\comDiffRescaledOne,\comDiffRescaledTwo\right)= & \min\{1-\posRescaledOne\comDiffRescaledOne,1-\posRescaledTwo\comDiffRescaledTwo\}\\
&\hspace{.2cm}-\min\left\{\min\{1-\posRescaledOne\comDiffRescaledOne,1-\posRescaledTwo\comDiffRescaledTwo\},\max\{(1-\posRescaledOne)\comDiffRescaledOne,(1-\posRescaledTwo)\comDiffRescaledTwo\}\right\}.
\end{align*}
Recall the integral representation
$$
\min\{x,y\}=\int_{0}^{+\infty}\Indens{t\le x}\Indens{t\le y}dt,
$$
which is valid for all $(x,y)\in\RR_+^2$. We may express $f$ as 
\begin{align*}
 f\left(\posRescaledOne,\posRescaledTwo,\comDiffRescaledOne,\comDiffRescaledTwo\right)&=\int_{0}^{+\infty}\Indens{t\le \min\{1-\posRescaledOne\comDiffRescaledOne,1-\posRescaledOne\comDiffRescaledOne\}}\left(1-\Indens{t\le \max\{(1-\posRescaledOne)\comDiffRescaledOne,(1-\posRescaledTwo)\comDiffRescaledTwo\}}\right)dt\\
 &=\int_0^1 \Indens{\max\{(1-\posRescaledOne)\comDiffRescaledOne,(1-\posRescaledTwo)\comDiffRescaledTwo\}\le t\le \min\{1-\posRescaledOne\comDiffRescaledOne,1-\posRescaledOne\comDiffRescaledOne\}}dt\\
 &=\int_0^1 \Indens{(1-\posRescaledOne)\comDiffRescaledOne\le t\le 1-\posRescaledOne\comDiffRescaledOne}\Indens{(1-\posRescaledTwo)\comDiffRescaledTwo\le t\le 1-\posRescaledTwo\comDiffRescaledTwo}dt\\
 &=\int_0^1 \Ind{\left\{\comDiffRescaledOne\le \min\left\{t/\overline{\posRescaledOne},\overline{t}/\posRescaledOne\right\}\right\}}\Ind{\left\{\comDiffRescaledTwo\le \min\left\{t/\overline{\posRescaledTwo},\overline{t}/\posRescaledTwo\right\}\right\}}dt,
 \end{align*}
 using the convention that $1/0=+\infty$, $\min\{x,+\infty\}=x$, and $\overline{x}:=1-x$ for all $x\in[0,1]$.
 Consequently, by integrating over $(\comDiffRescaledOne,\comDiffRescaledTwo)\in[0,1]^2$ and using Fubini's Theorem, we obtain
\begin{align*}
 \int_{ [0, 1]^2 } f\left(\posRescaledOne,\posRescaledTwo,\comDiffRescaledOne,\comDiffRescaledTwo\right)d\comDiffRescaledOne d\comDiffRescaledTwo&=\int_0^1 \min\left\{\frac{t}{\overline{\posRescaledOne}},\frac{\overline{t}}{\posRescaledOne}\right\}\min\left\{\frac{t}{\overline{\posRescaledTwo}},\frac{\overline{t}}{\posRescaledTwo}\right\} dt.
\end{align*}
Hence,~\eqref{eq:loosePairsFormula2} simplifies to become 
\begin{align*}
\frac{\numberMixedLoosePairs{\ellOne}{\ellTwo}}{n\comDiffMaxOne\comDiffMaxTwo(\ellOne-1)(\ellTwo-1)}=(1\pm o(1))\int_{ [0, 1]^3 }\min\left\{\frac{t}{\overline{\posRescaledOne}},\frac{\overline{t}}{\posRescaledOne}\right\}\min\left\{\frac{t}{\overline{\posRescaledTwo}},\frac{\overline{t}}{\posRescaledTwo}\right\} d\measurePos{\ellOne}(\posRescaledOne) d\measurePos{\ellTwo}(\posRescaledTwo) dt,
\end{align*}
where $\measurePos{\ellOne}$ and $\measurePos{\ellTwo}$ are the measures defined in~\eqref{def:measurePos}. Now, we observe that 
$$
n\comDiffMaxOne\comDiffMaxTwo(\ellOne-1)(\ellTwo-1)=(1\pm o(1))n^3,
$$
and so 
$$
\frac{\numberMixedLoosePairs{\ellOne}{\ellTwo}}{n^3}\tendsto{n}{+\infty} \int_0^1\covarianceFunction{\ellOne}(t)\covarianceFunction{\ellTwo}(t)dt\stackrel{Cl.\ref{claim:covarianceFunctionConstant}}{\ge} \int_{1/3}^{2/3}\frac{1}{4(\ellOne-1)(\ellTwo-1)}dt=\frac{1}{12(\ellOne-1)(\ellTwo-1)}>0,
$$
completing the proof of Lemma~\ref{lem:mix-loose-pair} when both $\ellOne$ and $\ellTwo$ are constant. 

Assume now that $\ellTwo$ is a constant, but $\ellOne=\ellOne(n)\to+\infty$ with $\ellOne=o(\log n)$. Then by Claim~\ref{claim:covarianceFunction} we have $\covarianceFunction{\ell}\to\covarianceFunction{\infty}$ in $L^2$, furthermore, we have $\|\covarianceFunction{\ellTwo}\|_2\le 1$, hence
$$
\left\vert\int_0^1\covarianceFunction{\ellTwo}(t)(\covarianceFunction{\infty}(t)-\covarianceFunction{\ellOne}(t))dt\right\vert \le \|\covarianceFunction{\ellTwo}\|_2\cdot \|\covarianceFunction{\infty}-\covarianceFunction{\ellOne}\|_2\to 0.
$$
This implies that
$$
\frac{\numberMixedLoosePairs{\ellOne}{\ellTwo}}{n^3}\tendsto{n}{+\infty} \int_0^1\covarianceFunction{\infty}(t)\covarianceFunction{\ellTwo}(t)dt\stackrel{Cl.\ref{claim:covarianceFunctionConstant}}{\ge} \frac{1}{2(\ellTwo-1)}\int_{1/3}^{2/3}\covarianceFunction{\infty}(t)dt\ge\frac{\log (3/2)}{3(\ellTwo-1)}>0,
$$
completing the claim for this case.

Similarly, if $\ellOne=\ellOne(n)\to+\infty$ and $\ellTwo=\ellTwo(n)\to+\infty$ with $\ellTwo\le \ellOne=o(\log n)$, then analogously to the previous case, we obtain  
$$
\frac{\numberMixedLoosePairs{\ellOne}{\ellTwo}}{n^3} \tendsto{n}{+\infty}\int_0^1\covarianceFunction{\infty}(t)^2dt=\frac{5}{6}-\frac{\pi ^2}{18}=0.2850\ldots,
$$
where we evaluated the integral using SageMath~\cite{Sage}. 
\end{proof}

\begin{remark}
The limits $\numberMixedLoosePairsLeadingConstant{\ellOne}{\ellTwo}:=\lim_{n\to+\infty}\numberMixedLoosePairs{\ellOne}{\ellTwo}n^{-3}$ can be computed explicitly based on their integral representation (and the help of SageMath) for specific choices of $\ellOne$ and $\ellTwo$; for instance, along the diagonal $\ellOne=\ellTwo$ we have
\begin{align*}
\numberMixedLoosePairsLeadingConstant{3}{3} = \frac{31}{48}\approx0.6458\: ; \quad \numberMixedLoosePairsLeadingConstant{4}{4} = \frac{130}{243}\approx0.5350\: ; \quad \numberMixedLoosePairsLeadingConstant{5}{5} = \frac{835}{1728}\approx0.4832\: ; \quad\dots
\end{align*} 
and similarly, we obtain $\numberMixedLoosePairsLeadingConstant{4}{3} = \frac{785}{1296}\approx0.6057$,  $\numberMixedLoosePairsLeadingConstant{5}{3} = \frac{335}{576}\approx0.5816$, and also $\numberMixedLoosePairsLeadingConstant{5}{4} = \frac{1339}{2592}\approx0.5166$. Further values are easily computed explicitly, however we do not believe that there exists a closed form expression for $\numberMixedLoosePairsLeadingConstant{\ellOne}{\ellTwo}$ in general.
\end{remark}

\subsection{Second moments}

Given any subset $T\subseteq [n]$, we define
$$
\IndicatorRVAPCentralised{T}:=\Ind{\{T\subseteq [n]_p\}}-p^{|T|}
$$
so $\EE(\IndicatorRVAPCentralised{T})=0$ for all $T\subset [n]$, and for any $3\le \ellOne=\ellOne(n)\le n$ we set 
$$
\RVnumberAPCentralised{\ellOne}{} := \RVnumberAP{\ellOne}{}-\EE(\RVnumberAP{\ellOne}{}) = \sum_{T\in\setAP{\ellOne}{}}\IndicatorRVAPCentralised{T}.
$$
First, we prove that the main contribution of the centred second moments comes from loose pairs, overlap pairs, or a combination of both.
\begin{lemma}\label{lem:covariance}
For $0<p=p(n)=o(1)$ and any $3\le \ellTwo=\ellTwo(n)\le  \ellOne=\ellOne(n)=o(n)$ we have 
$$
\EE(\RVnumberAPCentralised{\ellOne}{}\RVnumberAPCentralised{\ellTwo}{}) =(1\pm o(1))\left[\numberMixedLoosePairs{\ellOne}{\ellTwo} p^{\ellOne+\ellTwo-1}+\numberMixedOverlapPairs{\ellOne}{\ellTwo} p^{\ellOne}\right].
$$
In particular, we have 
$$
\sigma_{\ellOne}=(1\pm o(1))\sqrt{\numberLoosePairs{\ellOne} p^{2\ellOne-1}+\numberOverlapPairs{\ellOne}  p^{\ellOne}}.
$$
\end{lemma}

\begin{proof}
We observe that for any $r\in [\ellTwo]$ and $(T,T')\in\setAP{\ellOne}{}\times\setAP{\ellTwo}{}$ with $|T\cap T'|=r$, we have 
\begin{equation*}
\EE(\IndicatorRVAPCentralised{T}\IndicatorRVAPCentralised{T'})=\EE(\Indens{T\cup T'\subseteq[n]_p}-p^{\ellOne+\ellTwo})=p^{\ellOne+\ellTwo-r}-p^{\ellOne+\ellTwo}=(1\pm o(1))p^{\ellOne+\ellTwo-r}, 
\end{equation*}
while for any $(T,T')\in\setAP{\ellOne}{}\times\setAP{\ellTwo}{}$ with $|T\cap T'|=0$ we have
\begin{equation*}
\EE(\IndicatorRVAPCentralised{T}\IndicatorRVAPCentralised{T'})=\EE(\IndicatorRVAPCentralised{T})\EE(\IndicatorRVAPCentralised{T'})=0.
\end{equation*}
By distinguishing the size of the intersection we obtain
\begin{align*}
\EE(\RVnumberAPCentralised{\ellOne}{}\RVnumberAPCentralised{\ellTwo}{}) &=(1\pm o(1))\sum_{r =1}^{\ellTwo}\numberMixedPartitionPairs{\ellOne}{\ellTwo}{r}p^{\ellOne+\ellTwo-r},
\end{align*}
and recall that by definition $\numberMixedPartitionPairs{\ellOne}{\ellTwo}{1}=\numberMixedLoosePairs{\ellOne}{\ellTwo}$ and $\numberMixedPartitionPairs{\ellOne}{\ellTwo}{\ellTwo}=\numberMixedOverlapPairs{\ellOne}{\ellTwo}$. 

Therefore, we first consider the contribution of summands with $2\le r\le \lfloor 2\ellTwo/3\rfloor$. By the first estimate of Claim~\ref{claim:numberOtherPairs} we have 
$$
\sum_{r =2}^{\lfloor 2\ellTwo/3\rfloor}\numberMixedPartitionPairs{\ellOne}{\ellTwo}{r}p^{\ellOne+\ellTwo-r}=O(1)\cdot \sum_{r =2}^{\lfloor 2\ellTwo/3\rfloor}n^2\ellOne (\ellTwo)^2 p^{\ellOne+\ellTwo-r}=O( n^2 \ellOne (\ellTwo)^3 p^{\ellOne+\ellTwo/3})=o(\numberMixedOverlapPairs{\ellOne}{\ellTwo}p^\ellOne),
$$
where for the last estimate, we recall that $\numberMixedOverlapPairs{\ellOne}{\ellTwo}=\Theta(n^2(\ellOne-\ellTwo+1)(\ellTwo)^{-1})$ by Corollary~\ref{obs:numberOverlapPair}, and observe that $\tfrac{\ellOne(\ellTwo)^4p^{\ellTwo/3}}{\ellOne-\ellTwo+1}=o(1)$ for all constellations of $\ellOne$ and $\ellTwo$, since $p=o(1)$.

Next, we consider the contribution of summands with $\lfloor 2\ellTwo/3\rfloor+1\le r\le \ellTwo-1$. By the second estimate of Claim~\ref{claim:numberOtherPairs} we obtain
\begin{align*}
\sum_{r =\lfloor 2\ellTwo/3\rfloor+1}^{\ellTwo-1}\numberMixedPartitionPairs{\ellOne}{\ellTwo}{r}p^{\ellOne+\ellTwo-r}&=O(1)\cdot\sum_{r =\lfloor 2\ellTwo/3\rfloor+1}^{\ellTwo-1} n^2(\ellOne-r+1)(\ellTwo)^{-1}(\ellTwo-r+1) p^{\ellOne+\ellTwo-r}\\
    &=O(n^2(\ellOne-\ellTwo+1)(\ellTwo)^{-1}p^\ellOne)\cdot \sum_{i=0}^{\lceil\ellTwo/3\rceil-2}(i+2)^2 p^{i+1}\\
    &=o(\numberMixedOverlapPairs{\ellOne}{\ellTwo}p^\ellOne),
\end{align*}
since the last sum is of order $O(p)=o(1)$.

Hence, the main contribution to $\EE(\RVnumberAPCentralised{\ellOne}{}\RVnumberAPCentralised{\ellTwo}{}) $ comes from the summands for $r=1$ and $r=\ellTwo$, i.e.\ we have
$$
\EE(\RVnumberAPCentralised{\ellOne}{}\RVnumberAPCentralised{\ellTwo}{})=(1\pm o(1))\left[\numberMixedLoosePairs{\ellOne}{\ellTwo}p^{\ellOne+\ellTwo-1}+ \numberMixedOverlapPairs{\ellOne}{\ellTwo}p^{\ellOne}\right],
$$
as claimed by the first statement. As for the second statement, we recall that by definition $\numberMixedLoosePairs{\ellOne}{\ellOne}=\numberLoosePairs{\ellOne}$ and $\numberMixedOverlapPairs{\ellOne}{\ellOne}=\numberOverlapPairs{\ellOne}=\numberAP{\ellOne}$.
\end{proof}

For any $3\le \ellTwo=\ellOne(n)<\ellOne=\ellOne(n)\le n$ we define 
\begin{align}\label{def:LimitingCorrelation}
\covarianceLeadingConstant{\ellOne}{\ellTwo} :=\lim_{n\to+\infty}\frac{\EE(\RVnumberAPCentralised{\ellOne}{}\RVnumberAPCentralised{\ellTwo}{})}{\sigma_{\ellOne}\sigma_{\ellTwo}}
\end{align}
and observe that $0\le \covarianceLeadingConstant{\ellOne}{\ellTwo}\le 1$, by the FKG inequality and the Cauchy-Schwarz inequality. The following proof shows implicitly that $\covarianceLeadingConstant{\ellOne}{\ellTwo}$ is well-defined, i.e.\ the limit in~\eqref{def:LimitingCorrelation} exists.
\begin{lemma}\label{lem:kappa}
Let $0<p=p(n)=o(1)$ and $3\le\ellTwo=\ellTwo(n)<\ellOne=\ellOne(n)=o(n)$. 
\begin{enumerate}[(a)]
\item If $np^{\ellOne-1}\ellOne\to 0$, then 
$$
\covarianceLeadingConstant{\ellOne}{\ellTwo}=0;
$$
\item if $np^{\ellOne-1}\ellOne\to c\in\RR_+$, then 
$$
0<\covarianceLeadingConstant{\ellOne}{\ellTwo}< 1;
$$
\item if $np^{\ellOne-1}\ellOne\to+\infty$ and $\ellTwo$ is a constant, then 
$$
0<\covarianceLeadingConstant{\ellOne}{\ellTwo}<1;
$$
\item if $np^{\ellOne-1}\ellOne\to+\infty$ and $\ellTwo=\ellTwo(n)\to+\infty$, then 
$$
\covarianceLeadingConstant{\ellOne}{\ellTwo}=1.
$$
\end{enumerate} 
\end{lemma}
\begin{proof}
By Lemma~\ref{lem:covariance}, we have 
\begin{align*}
\frac{\left[\EE (\RVnumberAPCentralised{\ellOne}{}\RVnumberAPCentralised{\ellTwo}{})\right]^2}{\sigma_{\ellOne}^2\sigma_{\ellTwo}^2}&=(1\pm o(1))\frac{\left[\numberMixedLoosePairs{\ellOne}{\ellTwo}p^{\ellOne+\ellTwo-1}+ \numberMixedOverlapPairs{\ellOne}{\ellTwo}p^{\ellOne}\right]^2}{\left[\numberLoosePairs{\ellOne}p^{2\ellOne-1}+ \numberAP{\ellOne}p^{\ellOne}\right]\cdot \left[\numberLoosePairs{\ellTwo}p^{2\ellTwo-1}+ \numberAP{\ellTwo}p^{\ellTwo}\right]}.
\end{align*}

First assume that $np^{\ellOne-1}\ellOne\to 0$, then we have
$$
\numberLoosePairs{\ellOne}p^{2\ellOne-1}=\Theta(n^3p^{2\ellOne-1})= o(n^2\ellOne^{-1}p^{\ellOne}) =o(\numberAP{\ellOne}p^{\ellOne}),
$$
by Claim~\ref{obs:numberAP} and Lemma~\ref{lem:mix-loose-pair}. Consequently we have 
$$
\frac{\left[\EE (\RVnumberAPCentralised{\ellOne}{}\RVnumberAPCentralised{\ellTwo}{})\right]^2}{\sigma_{\ellOne}^2\sigma_{\ellTwo}^2}=O(1)\cdot\left[\frac{\left(\numberMixedLoosePairs{\ellOne}{\ellTwo}p^{\ellOne+\ellTwo-1}\right)^2}{\numberAP{\ellOne}p^{\ellOne}\cdot\numberLoosePairs{\ellTwo}p^{2\ellTwo-1}}+\frac{\left(\numberMixedOverlapPairs{\ellOne}{\ellTwo}p^{\ellOne}\right)^2}{\numberAP{\ellOne}p^{\ellOne}\cdot \numberAP{\ellTwo}p^{\ellTwo}}\right]. 
$$
Furthermore, using Claim~\ref{obs:numberAP} and Lemma~\ref{lem:mix-loose-pair} we obtain
$$
\frac{\left(\numberMixedLoosePairs{\ellOne}{\ellTwo}p^{\ellOne+\ellTwo-1}\right)^2}{\numberAP{\ellOne}p^{\ellOne}\cdot\numberLoosePairs{\ellTwo}p^{2\ellTwo-1}}=O(1)\cdot np^{\ellOne-1}\ellOne=o(1),
$$
and similarly, from Claim~\ref{obs:numberAP} and Corollary~\ref{obs:numberOverlapPair} we deduce
$$
\frac{\left(\numberMixedOverlapPairs{\ellOne}{\ellTwo}p^{\ellOne}\right)^2}{\numberAP{\ellOne}p^{\ellOne}\cdot \numberAP{\ellTwo}p^{\ellTwo}}=O(1)\cdot \frac{p^{\ellOne-\ellTwo}\ellOne(\ellOne-\ellTwo+1)^2}{\ellTwo}=o(1).
$$
Hence, letting $n\to+\infty$ we obtain 
$$
\covarianceLeadingConstant{\ellOne}{\ellTwo}=\lim_{n\to+\infty}\frac{\EE(\RVnumberAPCentralised{\ellOne}{}\RVnumberAPCentralised{\ellTwo}{})}{\sigma_{\ellOne}\sigma_{\ellTwo}}\le 0,
$$
as claimed since we already argued that $\covarianceLeadingConstant{\ellOne}{\ellTwo}\ge 0$ by the FKG inequality.

 On the other hand, if $np^{\ellOne-1}\ellOne\to c\in\RR_+$, then Claim~\ref{obs:numberAP}, Corollary~\ref{obs:numberOverlapPair}, and Lemma~\ref{lem:mix-loose-pair} imply
 $$
 \numberAP{\ellTwo}p^{\ellTwo}=\Theta(n^2(\ellTwo)^{-1}p^{\ellTwo})=o(n^3p^{2\ellTwo-1})  =o(\numberLoosePairs{\ellTwo}p^{2\ellTwo-1}),
 $$
  and 
  $$
   \numberMixedOverlapPairs{\ellOne}{\ellTwo}p^{\ellOne} =\Theta(n^2(\ellOne-\ellTwo+1)(\ellTwo)^{-1}p^{\ellOne})=o(n^3p^{\ellOne+\ellTwo-1})=o(\numberMixedLoosePairs{\ellOne}{\ellTwo}p^{\ellOne+\ellTwo-1}).
    $$
   Thus we obtain
\begin{align}\label{eq:inbetweenCovariance}
\frac{\left[\EE (\RVnumberAPCentralised{\ellOne}{}\RVnumberAPCentralised{\ellTwo}{})\right]^2}{\sigma_{\ellOne}^2\sigma_{\ellTwo}^2} &=(1\pm o(1))\cdot\frac{\left(\numberMixedLoosePairs{\ellOne}{\ellTwo}p^{\ellOne+\ellTwo-1}\right)^2}{\left[\numberLoosePairs{\ellOne}p^{2\ellOne-1}+\numberAP{\ellOne}p^{\ellOne}\right]\cdot \numberLoosePairs{\ellTwo}p^{2\ellTwo-1}}.
\end{align}

Now let $\varphi_\ellOne:=\covarianceFunction{\ellOne}$ if $\ellOne$ is a constant, and $\varphi_\ellOne:=\covarianceFunction{\infty}$ if $\ellOne=\ellOne(n)\to+\infty$; and define $\varphi_\ellTwo$ analogously. We note that both $\varphi_\ellOne$ and $\varphi_\ellTwo$ are $L^2$-integrable. Next, we take the limit $n\to +\infty$ in~\eqref{eq:inbetweenCovariance} and note that Lemma~\ref{lem:mix-loose-pair} implies
$$
\covarianceLeadingConstant{\ellOne}{\ellTwo}^2=\lim_{n\to+\infty}\frac{\left[\EE (\RVnumberAPCentralised{\ellOne}{}\RVnumberAPCentralised{\ellTwo}{})\right]^2}{\sigma_{\ellOne}^2\sigma_{\ellTwo}^2}=\frac{1}{1+\gamma}\cdot \frac{\langle \varphi_\ellOne,\varphi_\ellTwo\rangle^2}{\|\varphi_\ellOne\|_2^2\|\varphi_\ellTwo\|_2^2}
$$
where
$$
\gamma=\gamma(c,\ellOne):=\lim_{n\to+\infty}\frac{\numberAP{\ellOne}p^{\ellOne}}{\numberLoosePairs{\ellOne}p^{2\ellOne-1}} =
\begin{cases} 
\frac{\ellOne}{2(\ellOne-1)c\|\varphi_\ellOne\|_2^2} \quad \mbox{if $ \ell $ is finite} \\ \frac{1}{2 c\|\varphi_\infty\|_2^2} \quad\quad \ \mbox{ otherwise}
\end{cases}.
$$
In particular, the Cauchy-Schwarz inequality implies 
$$
\covarianceLeadingConstant{\ellOne}{\ellTwo}^2 \le \frac{1}{1+\gamma}<1,
$$
since $\gamma>0$. On the other hand, Lemma~\ref{lem:mix-loose-pair} also guarantees that $\langle \varphi_\ellOne,\varphi_\ellTwo\rangle>0$ and this implies
$$
\covarianceLeadingConstant{\ellOne}{\ellTwo}>0,
$$
completing the proof for the case $np^{\ellOne-1}  \ellOne\to c\in\RR_+$.

Assume now that $np^{\ellOne-1}  \ellOne\to+\infty$, then $\numberAP{\ellTwo}p^{\ellTwo}=o(\numberLoosePairs{\ellTwo}p^{2\ellTwo-1})$, $\numberAP{\ellOne}p^{\ellOne}=o(\numberLoosePairs{\ellOne}p^{2\ellOne-1})$, and $\numberMixedOverlapPairs{\ellOne}{\ellTwo}p^{\ellOne} =\Theta(n^2(\ellOne-\ellTwo+1)(\ellTwo)^{-1}p^{\ellOne})=o(n^3p^{\ellOne+\ellTwo-1})=o(\numberMixedLoosePairs{\ellOne}{\ellTwo}p^{\ellOne+\ellTwo-1})$, by Claim~\ref{obs:numberAP}, Corollary~\ref{obs:numberOverlapPair}, and Lemma~\ref{lem:mix-loose-pair}. Therefore, we obtain
$$
\covarianceLeadingConstant{\ellOne}{\ellTwo}^2=\lim_{n\to+\infty}\frac{\left(\numberMixedLoosePairs{\ellOne}{\ellTwo}\right)^2}{\numberLoosePairs{\ellOne}\cdot \numberLoosePairs{\ellTwo}}=\frac{\langle \varphi_\ellOne,\varphi_\ellTwo\rangle^2}{\|\varphi_\ellOne\|_2^2\|\varphi_\ellTwo\|_2^2}
$$
from Lemmas~\ref{lem:covariance} and~\ref{lem:mix-loose-pair}, using the notation of $\varphi_\ellOne$ and $\varphi_\ellTwo$ as in the previous case. As before, we observe that $\langle \varphi_\ellOne,\varphi_\ellTwo\rangle>0$ and this implies
$$
\covarianceLeadingConstant{\ellOne}{\ellTwo}>0.
$$
It remains to distinguish two cases: first, if $\ellTwo=\ellTwo(n)\to+\infty$, then also $\ellOne=\ellOne(n)\to+\infty$ and thus $\varphi_\ellOne=\varphi_\ellTwo=\covarianceFunction{\infty}$, but then clearly $\langle \covarianceFunction{\infty},\covarianceFunction{\infty}\rangle=\|\covarianceFunction{\infty}\|_2^2$, so $\covarianceLeadingConstant{\ellOne}{\ellTwo}=1$. 

On the other hand, if $\ellTwo$ is a constant, then we observe that $\covarianceFunction{\ellTwo}$ and $\covarianceFunction{\infty}$ are linearly independent in $L^2$. To see this, let $\eps=\eps(\ellTwo)>0$ be a sufficiently small constant, and observe that $\covarianceFunction{\infty}(x)^2\le (2x+ x\log (1/x))^2\le 9x $ for all $x\le \eps$ implying 
$$
\int_0^\eps \covarianceFunction{\infty}(x)^2dx\le \frac{9}{2}\eps^2;
$$
however, by Lemma~\ref{claim:covarianceFunctionConstant}~(b), for sufficiently small $\eps>0$, we have $\covarianceFunction{\ellTwo}(x)\ge 1/(\ellTwo-1)$ and thus
$$
\int_0^\eps \covarianceFunction{\ellTwo}(x)^2dx\ge \frac{1}{(\ellTwo-1)^2}\eps.
$$
Consequently, for any sufficiently small constant $\eps>0$ we obtain 
$$
\int_0^\eps \left(\frac{\covarianceFunction{\ellTwo}(x)}{\|\covarianceFunction{\ellTwo}\|}\right)^2dx\ge \frac{1}{(\ellTwo-1)^2\|\covarianceFunction{\ellTwo}\|^2}\eps> \frac{9}{2\|\covarianceFunction{\infty}\|^2}\eps^2\ge  \int_0^\eps \left(\frac{\covarianceFunction{\infty}(x)}{\|\covarianceFunction{\infty}\|}\right)^2dx,
$$
and so the functions $\covarianceFunction{\ellTwo}$ and $\covarianceFunction{\infty}$ are not linearly dependent in $L^2$, as claimed. Consequently, the Cauchy-Schwarz inequality is a strict inequality and we obtain 
$$
\covarianceLeadingConstant{\ellOne}{\ellTwo}=\frac{\langle \varphi_\ellOne,\varphi_\ellTwo\rangle^2}{\|\varphi_\ellOne\|_2^2\|\varphi_\ellTwo\|_2^2}<1,
$$
completing the proof.
\end{proof}

\section{Univariate fluctuations: proof of Theorem~\ref{thm:mainUnivariate}}\label{Sec:Univariate}

In this section we focus on univariate fluctuations of $\numberAP{\ell}$, i.e.\ we prove the two statements of Theorem~\ref{thm:mainUnivariate}. First we treat the Poisson regime, where the result follows directly from an application of the Chen-Stein method and the preliminary computations performed in Section~\ref{Sec:Preliminaries} (with $\ell'=\ell$). Likewise, the Gaussian approximation is a consequence of a classical normality criterion.

\subsection{Poisson regime: proof of Theorem~\ref{thm:mainUnivariate}\eqref{thm:mainUnivariatePoisson}}\label{Sec:Poisson}
We start by introducing the notion of a dependency graph. We emphasize the fact that this definition is the one that fits our purpose, and that there can be many other such notions (see e.g. \cite{FerayWeighted,JansonLuczakRucinski}). 
\begin{definition}\label{def:depGraph}
Let  $ (Y_i)_{1 \leq i \leq N} $ be a sequence of random variables (on a common probability space). A (simple) graph $G=(V,E)$  with vertex set $V=[N]$ is called a \emph{dependency graph} for $(Y_i)_{i\in[N]}$ if and only if for all disjoint subsets $U,U'\subseteq V$ with $E(U,U')=\emptyset$ we have
$$
(Y_i)_{i\in U} \text{ is independent of } (Y_i)_{i\in U'},
$$
where $E(U,U'):=\{(i,j)\in E\colon i\in U\text{ and } j\in U'\}$ denotes the set of edges between $U$ and $U'$.
We denote the neighbourhood of a vertex $i\in [N]$ by $\cN(i):=\cN_G(i):=\{j\in U\colon (i,j)\in E\}$ and let $\widehat{\cN}(i):=\cN(i)\cup \{i\}$.
\end{definition}
The dependency graph relevant to this paper is the following: given $3\le\ell'=\ell'(n)\le \ell=\ell(n)\le n$ we consider the graph 
\begin{equation}\label{eq:depGraph}
\depGraph{\ell,\ell'}=\depGraph{\ell,\ell'}(n):=\left(\setAP{\ell}{}\cup\setAP{\ell'}{}, \left\{(T,T')\in \left(\setAP{\ell}{}\cup\setAP{\ell'}{}\right)^2\colon |T\cap T'|\ge 1 \right\}\right).
\end{equation}
In other words, the vertices represent APs and edges indicate that the corresponding APs intersect. Clearly, $\depGraph{\ell,\ell'}$ is a dependency graph of the family $\left(\Indens{T\in [n]_p}\right)_{T\in \setAP{\ell}{}\cup\setAP{\ell'}{}}$.

We define the following two quantities associated with a dependency graph $G$ of $(Y_i)_{1 \leq i \leq N} $:
\begin{align}\label{Eq:SteinQuantities}
\begin{split}
\SteinControl{1}{G} :=& \sum_{i = 1}^N \sum_{j \in \neighbourhoodClosed{i}{G}} \mathbb{E}( Y_i ) \mathbb{E}( Y_j ), \\
\SteinControl{2}{G} :=& \sum_{i = 1}^N \sum_{j \in \neighbourhood{i}{G} } \mathbb{E}( Y_i Y_j ).
\end{split}
\end{align}
We use a variant of the Chen-Stein method due to Arratia, Goldstein, and Gordon~\cite{ArratiaGoldsteinGordon} (in a slightly simplified form).
\begin{theorem}[Theorem~1 in~\cite{ArratiaGoldsteinGordon}]\label{thm:AGG}
	Let $(Y_i)_{1 \leq i \leq N}$ be Bernoulli random variables of expectation $p_i := \mathbb{E}( Y_i ) > 0$.  Set
	\begin{align*}
	S_N := \sum_{i=1}^N Y_i, \qquad \textrm{ and } \qquad \ \zeta := \mathbb{E}( S_N ) = \sum_{i=1}^N p_i .
	\end{align*}
	Let $G$ be a dependency graph of $(Y_i)_{1 \leq i \leq N}$, and $\SteinControl{1}{G}$, $\SteinControl{2}{G}$ as in \eqref{Eq:SteinQuantities}. Let $\cY$ be a Poisson random variable with mean $\mathbb{E}(\cY) := \zeta$. Then, for any $U \subset \mathbb{N}$,
	\begin{align*}
	\left\vert \mathbb{P}( S_N \in U ) - \mathbb{P}(\cY \in U) \right\vert \leq \SteinControl{1}{G} + \SteinControl{2}{G}.
	\end{align*}
\end{theorem}

\begin{remark}
	The theorem given in~\cite{ArratiaGoldsteinGordon} uses an additional quantity $ \SteinControl{3}{G} $ given by
	\begin{align*}
\SteinControl{3}{G} :=& \sum_{i = 1}^N \mathbb{E}\!\left( \left\vert \mathbb{E}\left( Y_i - p_i \cond (Y_j)_{j \not\in \neighbourhoodClosed{i}{G}}\right)  \right\vert  \right)
	\end{align*}
	but due to using a more restrictive notion of dependency graphs, we always have $\SteinControl{3}{G}=0$.
\end{remark}

\begin{proof}[Proof of Theorem~\ref{thm:mainUnivariate}\eqref{thm:mainUnivariatePoisson}]
We fix any $3\le \ell=\ell(n)\le n$  and aim to apply Theorem~\ref{thm:AGG} to the family $\left(\Indens{T\in [n]_p}\right)_{T\in\setAP{\ell}{}}$. The corresponding dependency graph $\depGraph{\ell}$ was defined in~\eqref{eq:depGraph}. Clearly, for any $T\in\setAP{\ell}{}$ we have  $\EE(\Indens{T\subseteq [n]_p})=p^\ell$ and thus 
\begin{align*}
\SteinControl{1}{\depGraph{\ell}} =& \sum_{T\in\setAP{\ell}{}} \sum_{T' \in \neighbourhoodClosed{T}{\depGraph{\ell}}} \mathbb{E}(\Indens{T\subseteq [n]_p}) \mathbb{E}(\Indens{T'\subseteq [n]_p}) =p^{2\ell} \sum_{r =1}^{\ell}\numberPartitionPairs{\ell}{r}=O(n^3p^{2\ell}),
\end{align*}
where the last equality holds due to Corollary~\ref{obs:numberOverlapPair}, Claim~\ref{claim:numberOtherPairs}, and Lemma~\ref {lem:mix-loose-pair}.

Next, we note that $\EE( \Indens{T\subseteq [n]_p} \Indens{T'\subseteq [n]_p} )=p^{2\ell-r}$ for all $1\le r\le \ell-1$ and  $(T,T')\in\setPartitionPairs{\ell}{r}$. Thus, we obtain
\begin{align*}
\SteinControl{2}{\depGraph{\ell}} = \sum_{T\in\setAP{\ell}{}} \sum_{T' \in \neighbourhood{T}{G} } \mathbb{E}\left(\Indens{T\subseteq [n]_p} \Indens{T'\subseteq [n]_p}\right)&=\sum_{r =1}^{\ell-1}\numberPartitionPairs{\ell}{r}p^{2\ell-r}\\ 
&=O(1)\cdot\left[ n^3p^{2\ell-1}+n^2\ell^3p^{\ell+1} \right],
\end{align*}
where the last estimate holds due to Claim~\ref{claim:numberOtherPairs} and Lemma~\ref{lem:mix-loose-pair}.

Combining these two bounds and using the assumption $n^2p^{\ell}/(\ell-1)\to c$ for some $c\in\RR_+$ yields
$$
\SteinControl{1}{\depGraph{\ell}}+\SteinControl{2}{\depGraph{\ell}}=O(n^{-1+2/\ell}+n^{-2/\ell})=o(1).
$$
The same bound holds when $\ell\rightarrow +\infty$, $\ell=o(\log n)$ and $p\ell^4\rightarrow 0$.
Thus Theorem~\ref{thm:AGG} is applicable for the family $\left(\Indens{T\in [n]_p}\right)_{T\in\setAP{\ell}{}}$ and shows that for all $U\subseteq \NN$ we have
$$
\left|\PP\left(\sum_{T \in \setAP{\ell}{} }\Indens{T\subseteq [n]_p}\in U\right) -\PP\left(\Po\left(\lambda\right)\in U\right)\right|\le o(1),
$$
where 
$$
\lambda:=\lim_{n\to+\infty}\sum_{ T \in \setAP{\ell}{} }\EE(\Indens{T\subseteq [n]_p})\stackrel{Cl.~\ref{obs:numberAP}}{=}\lim_{n\to+\infty}(1\pm o(1))\frac{n^2p^{\ell}}{2(\ell-1)}=
c/2
$$
completing the proof of Theorem~\ref{thm:mainUnivariate}\eqref{thm:mainUnivariatePoisson}.
\end{proof}

\begin{remark}\label{Rk:ImprovedPoissonApproximation}
If we do not suppose the assumption of Theorem~\ref{thm:mainUnivariate}\eqref{thm:mainUnivariatePoisson}, namely that $ n^2 p^\ell \ell^{-1} = O(1)$, we still have a Poisson approximation with $ \Po(\lambda_n) $ where $ \lambda_n := n^2 p^\ell \ell^{-1} $ provided that $ \SteinControl{1}{\depGraph{\ell}}+\SteinControl{2}{\depGraph{\ell}}=o(1) $. This is the case if $ n^3 p^{2\ell - 1} \to 0 $ and $ n^2  p^{\ell + 1} \ell^3\to 0  $, which is equivalent in the first case to $ p \ll n^{-3/(2\ell - 1) } $, and in the second case to $ p \ll n^{- 2/(\ell + 1) }\ell^{- 3/(\ell + 1) }  $. It is well known that a Poisson random variable with diverging parameter converges in distribution (after rescaling) to a Gaussian, hence, this case shows that we have a Gaussian regime for the range $ n^{-2/\ell } \ell^{1/\ell} \ll p \ll \min\{ n^{-3/(2\ell - 1) },  n^{- 2/(\ell + 1) }\ell^{- 3/(\ell + 1) } \} $. 
\end{remark}

\subsection{Gaussian regime: proof of Theorem~\ref{thm:mainUnivariate}\eqref{thm:mainUnivariateNormal}}\label{Sec:GaussianRegime}
For the normal approximation we apply a criterion due to Janson~\cite{JansonCLTwithCumulants}, which was then refined by Mikhailov~\cite{Mikhailov}. This normality criterion is based on controlling mixed cumulants of sum of random variables by means of an associated dependency graph. We follow the notation of~\cite{JansonLuczakRucinski}.  

\begin{theorem}[e.g.\ Theorem~6.21 in~\cite{JansonLuczakRucinski}]\label{Thm:MikhailovNormality}
    Let $(X_{i, n})_{1\le i \le N_n}$ be a family of random variables with dependency graph $\Gamma_n$ (as defined in Definition~\ref{def:depGraph}) and suppose that there exist constants $ \{C_r\}_{r\in\NN} $ independent of $n$, and quantities $M_n$ and $Q_n$ such that 
    \begin{align}\label{eq:FirstCondition}
    \EE\!\left( \sum_{i = 1}^{N_n} \vert X_{i, n} \vert \right)\le M_n,
    \end{align}
    and for all $ V $ of constant size (i.e.\ $ |V| $ is independent of $n$), we have 
    \begin{align}\label{eq:SecondCondition}
    \sum_{i \in \cN(V)} \EE\!\left( \vert X_{i, n}\vert \big\vert (X_{j, n})_{j \in V } \right) \leq C_{|V|} Q_n,
    \end{align}
    where $ \cN(V) := \cup_{i \in V} \cN(i) $ as in Definition~\ref{def:depGraph}.

  Let $S_n := \sum_{i=1}^{N_n} X_{i, n}$ and $\sigma_n^2 := \VV(S_n)$. If there exists an $s>2$ such that
    \begin{align}\label{Eq:MikhailovCriteria}
     \frac{M_n}{\sigma_n}\left(\frac{Q_n}{\sigma_n}\right)^{s-1} \tendsto{n}{+\infty} 0 
    \end{align}
    then, we have 
    \begin{align*}
      \frac{S_n - \EE(S_n)}{\sigma_n} \stackrel{d}{\tendsto{n}{+\infty}} \No(0,1).
    \end{align*}
\end{theorem}
Note that the proof of Theorem~\ref{Thm:MikhailovNormality} shows that the assumption~\eqref{Eq:MikhailovCriteria} becomes weaker as $s$ increases. However, we will see that for this application it is satisfied for any $s>0$.

\begin{proof}[Proof of Theorem~\ref{thm:mainUnivariate}\eqref{thm:mainUnivariateNormal}]
In the setting of $ \ell $-APs we have $S_n:=\RVnumberAP{\ell}{}$, and note that by Claim~\ref{obs:numberAP} we have $\EE(\RVnumberAP{\ell}{})=p^\ell \numberAP{\ell}$, i.e.\ we may choose $M_n:=p^\ell \numberAP{\ell}$.

Now, for $ V \subset \setAP{\ell}{} $, let $\Lambda(V) := \cup_{T\in V} T \subset [n] $ be the set of points covered by APs in $V$. We write $Z(V)$ for the LHS of~\eqref{eq:SecondCondition} and observe that
\begin{align*}
Z(V)=\sum_{T \in \cN(V) } \EE\!\left(  \prod_{a \in T} \Xi_a \bigg\vert ( \Xi_k )_{k \in \Lambda(V) } \right) &= \sum_{T \in \cN(V) } p^{ \ell - \vert T \cap \Lambda(V) \vert } \prod_{a \in T \cap \Lambda(V) } \Xi_a  \\
& \leq \sum_{T \in \cN(V) }   p^{ \ell - \vert T \cap \Lambda(V) \vert }
\end{align*}
as $ \Xi_a$ takes values in $\{0, 1 \} $. First, we consider APs $T\in\cN(V)$ in ``loose configurations'', i.e.\ $|T \cap \Lambda(V) | = 1$. Note that there are at most $O_{|V|}(n\ell)$ of these $T$ and the contribution to $Z(V)$ of each of them is $p^{\ell-1}$. On the other hand, there are at most $O_{|V|}(\ell^4)$ APs $T\in \cN(V)$ with $|T \cap \Lambda(V) | \ge 2$, and trivially, each of their contribution to $Z(V)$ is upper bounded by $1$. Together this means that there exist constants $\{C_r\}_{r\in\NN}$ and we may choose $ Q_n:= n p^{\ell - 1}\ell + \ell^4$ such that  $Z(V)\le C_{|V|}Q_n$ for all $V\subset \setAP{\ell}{}$ of constant size.

Recall that Lemma~\ref{lem:covariance} gives $ \sigma_n=(1\pm o(1))\sqrt{\numberLoosePairs{ \ell } p^{2\ell -1} + \numberOverlapPairs{\ell } p^{\ell } } $ with $ \numberOverlapPairs{\ell } = \numberAP{\ell } = \Theta(n^2\ell^{-1}) $ and $ \numberLoosePairs{ \ell } = \Theta(n^3) $ by Claim~\ref{obs:numberAP} and Lemma~\ref{lem:mix-loose-pair}, respectively. Thus we have $\sigma_n=\Theta(\sqrt{n^2p^{\ell}\ell^{-1}\left(1+np^{\ell-1}\ell\right)})$ and we distinguish two cases:

If $ np^{\ell-1}\ell \ge 10$, then $M_n/\sigma_n=O(n^{1/2}p^{1/2}\ell^{-1})$ and $Q_n/\sigma_n\le np^{\ell-1}\ell^5/\sigma_n=O(n^{-1/2}p^{-1/2}\ell^ 5)$. Thus, for any $s>2$, we have 
$$
 \frac{M_n}{\sigma_n}\left(\frac{Q_n}{\sigma_n}\right)^{s-1}=O\left((np)^{-(s-2)/2}\ell^{5(s-1)-1}\right)=o(1),
$$
since $np\to +\infty$ polynomially in $n$ and $\ell=o(\log n)$.

Otherwise, we have $ np^{\ell-1}\ell \le 10$ which implies $M_n/\sigma_n=O(np^{\ell/2}\ell^{-1/2})$ and $Q_n/\sigma_n=O(n^{-1}p^{-\ell/2}\ell^{9/2})$. Consequently, for any $s>2$, we obtain 
$$
\frac{M_n}{\sigma_n}\left(\frac{Q_n}{\sigma_n}\right)^{s-1}=O\left((n^2p^\ell)^{-(s-2)/2}\ell^{ (9s - 10)/2 }\right).
$$
Next, we recall that by Remark~\ref{Rk:ImprovedPoissonApproximation} we may additionally assume that $p$ is not too small, e.g.\ $p\ge \eps n^{-\max\{3/(2\ell-1),2/(\ell+1)\}}$ for any $\eps=\eps(n)>0$ with $\eps\to 0$. It remains to observe that when $\eps$ is decreasing sufficiently slowly this implies that $n^2p^\ell \gg e^{ \Omega(\log n)/\ell } $. Since $\ell=o(\log n)$, it follows that~\eqref{Eq:MikhailovCriteria} is satisfied and applying Theorem~\ref{Thm:MikhailovNormality} completes the proof of Theorem~\ref{thm:mainUnivariate}\eqref{thm:mainUnivariateNormal}. 
\end{proof}

\section{Bivariate fluctuations: proof of Theorem~\ref{thm:mainBivariate}}\label{Sec:GaussianRegimeBivariate}

For the rest of this Section, let $3\le \ell_2=\ell_2(n)< \ell_1=\ell_1(n)$ and $0<p=p(n)<1$ such that
\begin{align}
\label{eq:assumptionPUpper} p\ell_1^{9}&\tendsto{n}{+\infty} 0,\\
\label{eq:assumptionPLower} n^2p^{\ell_1}\ell_1^{-9}&\tendsto{n}{+\infty}+\infty,\\
\label{eq:assumptionLUpper} \frac{\ell_1}{\log n}&\tendsto{n}{+\infty} 0.
\end{align}
Our goal is to apply the method of moments (cf.\ Theorems~\ref{thm:MethodOfMoments} and~\ref{thm:CramerWoldDevice}), therefore we want to determine the asymptotics of the $k$-th moments $\EE\left[ \left(u_{\ell_1}\frac{\RVnumberAPCentralised{\ell_1}{}}{\sigma_{\ell_1}}+u_{\ell_2}\frac{\RVnumberAPCentralised{\ell_2}{}}{\sigma_{\ell_2}}\right)^k\right]$ for all $k\in \NN$ and $u_{\ell_1},u_{\ell_2}\in\RR$. (We recall that $\sigma_{\ell_i}=\sqrt{\EE(\RVnumberAPCentralised{\ell_i}{2})}$ denotes the standard deviation of $\RVnumberAP{\ell_i}{}$ for $i\in\{1,2\}$.) By definition we have 
\begin{align}\label{eq:kthMoment}
\EE\! \left[\! \left(\! u_{\ell_1}\frac{\RVnumberAPCentralised{\ell_1}{}}{\sigma_{\ell_1}} \!+\! u_{\ell_2}\frac{\RVnumberAPCentralised{\ell_2}{}}{\sigma_{\ell_2}}\!\right)^{\! \!  k} \right] \! =\!  \sum_{\mathbf{T}\in\left(\setAP{\ell_1}{}\cup \setAP{\ell_2}{}\right)^{\! k}} \!\!\!\left(\frac{u_{\ell_1}}{\sigma_{\ell_1}}\right)^{\! \! k_1(\mathbf{T})}\left(\frac{u_{\ell_2}}{\sigma_{\ell_2}}\right)^{\!\! k_2(\mathbf{T})}\EE\!\left( \prod_{T\in\mathbf{T}}\IndicatorRVAPCentralised{T}\!\right) 
\end{align}
where $k_i(\mathbf{T}):=\left|\{T\in\mathbf{T}\colon |T|=\ell_i\}\right|$, for $i\in\{1,2\}$, is the number of $\ell_i$-APs in $\mathbf{T}$.

\begin{remark}\label{Rk:BivariateToUnivariate}
Note that despite our assumption that $\ell_1\neq \ell_2$, our approach also includes the univariate scenario: for 
$3\le \ell=\ell(n)=o(\log n)$ and $0<p=p(n)<1$ such that $p\ell^{9}\to 0$ and $n^2p^{\ell}\ell^{-9}\to +\infty$, we obtain the $k$-th moment $\EE(\RVnumberAPCentralised{\ell}{k})$ by setting $\ell_2=\ell$, $\ell_1=2\ell$, $u_{\ell_2}=1$, and $u_{\ell_1}=0$.

Furthermore, we observe that in the univariate case the additional assumption~\eqref{eq:assumptionPLower} comes without loss of generality, since we already noticed in Remark~\ref{Rk:ImprovedPoissonApproximation} that $\RVnumberAPCentralised{\ell}{}\sigma_{\ell}^{-1}$ has a Gaussian limit if $n^2p^{\ell}\ell^{-1}\to+\infty$ but  $n^2p^{\ell}\ell^{-9}=O(1)$.
\end{remark}

\subsection{Main contribution to the moments}\label{SubSec:MainContribution}

In~\eqref{eq:kthMoment} we expressed the $k$-th moment of an arbitrary linear combination of $\RVnumberAPCentralised{\ell_1}{}$ and $\RVnumberAPCentralised{\ell_2}{}$ as a sum ranging over  $k$-tuples of APs, each of length $\ell_1$ or $\ell_2$. We will now show that for even $k$ the main contribution to this sum comes from $k$-tuples $\mathbf{T}=(T_1,\dots, T_k)$ with a certain matching structure, namely there exists a bijective self-inverse mapping $\match\colon [k]\to[k]$ without fixed point (we will call such permutation a \emph{matching}) such that $\mathbf{T}$ satisfies
\begin{equation}\label{eq:domTuples}
\forall i\in[k] \colon \quad T_i\cap T_{\match(i)}\neq\emptyset \quad\wedge\quad T_i\cap\left(\bigcup_{j\in[k]\setminus\{i,\match(i)\}}T_j\right)=\emptyset.
\end{equation}
We write $\mainContributionSetMatching{k}{\match}$ for the set of (ordered) $k$-tuples satisfying~\eqref{eq:domTuples} for a given matching $\match$, and observe that any two distinct sets $\mainContributionSetMatching{k}{\match}$ and $\mainContributionSetMatching{k}{\match'}$, $\match\neq\match'$, are disjoint and can be mapped bijectively onto each other. Thus let $\match^*$ be defined by 
\begin{align*}
\nu^*(2i-1)=2i, \quad\forall i\in[k/2], 
\end{align*}
and note that there are precisely $(k-1)!!$ many distinct matchings $\match$. 
 
  Let  $\mainContribution{k}$ denote the contribution of $k$-tuples in $\mainContributionSet{k}:=\dot{\bigcup}_{\match}\mainContributionSetMatching{k}{\match}$ to the $k\text{-th}$ moment $\EE\left[ \left(u_{\ell_1}\frac{\RVnumberAPCentralised{\ell_1}{}}{\sigma_{\ell_1}}+u_{\ell_2}\frac{\RVnumberAPCentralised{\ell_2}{}}{\sigma_{\ell_2}}\right)^k\right] $, and set $ \mainContributionSet{k} := \emptyset $ for $k$ odd. Then we let $\minorContributionSet{k}:=\left(\setAP{\ell_1}{}\cup\setAP{\ell_2}{}\right)^k\setminus \mainContributionSet{k}$ for all $k\in\NN$, and denote the contribution of $\minorContributionSet{k}$ by $\minorContribution{k}$. In other words, we have
\begin{align}\label{eq:split}
\EE \left[\left( u_{\ell_1}\frac{\RVnumberAPCentralised{\ell_1}{}}{\sigma_{\ell_1}}+u_{\ell_2}\frac{\RVnumberAPCentralised{\ell_2}{}}{\sigma_{\ell_2}}\right)^k\right]  & = \mainContribution{k} + \minorContribution{k}, 
\end{align}
where 
\begin{align*}
\mainContribution{k}:=& \sum_{\mathbf{T}\in \mainContributionSet{k} }\left(\frac{u_{\ell_1}}{\sigma_{\ell_1}}\right)^{k_1(\mathbf{T})}\left(\frac{u_{\ell_2}}{\sigma_{\ell_2}}\right)^{k_2(\mathbf{T})}\EE\left(\prod_{T\in\mathbf{T}}\IndicatorRVAPCentralised{T}\right) \\
\minorContribution{k}:= &\sum_{\mathbf{T}\in \minorContributionSet{k}}  \!\left(\frac{u_{\ell_1}}{\sigma_{\ell_1}}\right)^{k_1(\mathbf{T})}\left(\frac{u_{\ell_2}}{\sigma_{\ell_2}}\right)^{k_2(\mathbf{T})}\EE\left(\prod_{T\in\mathbf{T}}\IndicatorRVAPCentralised{T}\right).
\end{align*}
We observe that by the previous argument we may express $\mainContribution{k}$ as
\begin{align}\label{eq:domTerm}
\mainContribution{k}=(k-1)!!\sum_{\mathbf{T}\in\mainContributionSetMatching{k}{\match^*}}\prod_{i=1}^{k/2}\EE\!\left(\frac{u_{|T_{2i-1}|}\IndicatorRVAPCentralised{T_{2i-1}}}{\sigma_{|T_{2i-1}|}}\cdot\frac{u_{|T_{2i}|}\IndicatorRVAPCentralised{T_{2i}}}{\sigma_{|T_{2i}|}}\right).
\end{align}

\begin{lemma}\label{lem:domTerm}
Let $k \in 2\NN$ and $u_{\ell_1}, u_{\ell_2}\in \RR$, then we have 
\begin{align*}
\mainContribution{k}=(1\pm o(1))(k-1)!!\left[u_{\ell_1}^2+u_{\ell_2}^2+2u_{\ell_1}u_{\ell_2}\covarianceLeadingConstant{\ell_1}{\ell_2}\right]^{k/2}.
\end{align*}
\end{lemma}

\begin{proof}
We enumerate the $k$-tuples $\mathbf{T}=(T_1,\dots,T_k)\in\mainContributionSetMatching{k}{\match^*}$ in a specific order. Define for $ j \in \{1, 2\} $,
\begin{align*}
\Theta_j(\mathbf{T}) & := \{ i \in [k/2] : |T_{2 i - 1}| = |T_{2i}| = \ell_j \}.
\end{align*}
and $ \Theta_3(\mathbf{T}) := [k/2]\setminus ( \Theta_1 \cup \Theta_2) $. In other words, $\Theta_j(\mathbf{T})$ is the set of intersecting pairs of $\ell_j$-APs in $\mathbf{T}$, $j\in\{1,2\}$, and $\Theta_3(\mathbf{T})$ is the set of mixed intersecting pairs. 

Let $ \theta_i := |\Theta_i(\mathbf{T})| $ for $ i \in \{1, 2, 3 \} $ so that $0\le \theta_{1}\le k/2$, $0\le \theta_{2}\le k/2-\theta_{1}$ and $\theta_{3} =k/2-\theta_{1}-\theta_{2}$. We now consider the set $[k/2]$ as a set of distinct ``labels'', and partition $[k/2]$ into classes $\cP_1$, $\cP_2$, and $\cP_{3}$ of sizes $\theta_{1}$, $\theta_2$, and $\theta_{3}$, respectively. Note that there are precisely $\binom{k/2}{\theta_{1}, \theta_{2}, \theta_{3}}$ many choices for this. We proceed in rounds $i=1,\dots, k/2$ where  we distinguish three cases according to the $i$-th label: 
\begin{enumerate}[(a)]
\item if $i\in \cP_1$, then we  choose an integer $1\le m_i\le \ell_1$ and set 
\begin{align*}
\cM_i:=\left\{(T,T')\in\setAP{\ell_1}{2}\colon |T\cap T'|=m_i\right\} ;
\end{align*}
\item if $i\in \cP_2$, then we choose an integer $1\le m_i\le \ell_2$ and set 
\begin{align*}
\cM_i:=\left\{(T,T')\in\setAP{\ell_2}{2}\colon |T\cap T'|=m_i\right\} ;
\end{align*}
\item and  if $i\in \cP_3$, then we choose an integer $1\le m_i\le \ell_2$ and set 
\begin{align*}
\cM_i:=\left\{(T,T')\in(\setAP{\ell_1}{}\times \setAP{\ell_2}{})\cup (\setAP{\ell_2}{}\times \setAP{\ell_1}{})\colon|T\cap T'|=m_i \right\}.
\end{align*}
\end{enumerate}

In each case, note that some of the elements $(T,T')$ of $\cM_i$ might not be valid choices for $(T_{2i-1},T_{2i})$, as $T\cup T'$ may contain elements from $T_j$ for some $j\in[2i-2]$ (thus violating~\eqref{eq:domTuples} and the definition of $\mainContributionSet{k}$). Nonetheless, we claim that almost all of them are indeed valid. More formally, let $$\cM_i^*:=\left\{(T,T')\in\cM_i\colon (T\cup T')\cap \left(\bigcup_{j=1}^{2i-2}T_j\right)=\emptyset\right\}$$
and note that $\left|\bigcup_{j=1}^{2i-2}T_j\right|\le k\ell_1$ for all $i\in[k/2]$. Now, observe that we can express~\eqref{eq:domTerm} by
\begin{align}\nonumber
\mainContribution{k}&=(k-1)!!\sum_{\theta_{1}=0}^{k/2}\sum_{\theta_{2}=0}^{k/2-\theta_{1}}\binom{k/2}{\theta_{1},\theta_{2},\theta_{3}}u_{\ell_1}^{2\theta_{1}+\theta_{3}}u_{\ell_2}^{2\theta_{2}+\theta_{3}}\\
&\hspace{5cm}\times\prod_{i=1}^{k/2}\left(\sum_{m_i}\sum_{(T,T')\in\cM_i^*}\frac{\EE\left(\IndicatorRVAPCentralised{T}\IndicatorRVAPCentralised{T'}\right)}{\sigma_{|T|}\sigma_{|T'|}}\right).\label{eq:domTerm2}
\end{align}

\begin{claim}\label{claim:smallError}
For any $R\subseteq [n]$ of size at most $k\ell_1$ we have 
$$
\left|\left\{(T,T')\in\cM_i \colon (T\cup T')\cap R\neq\emptyset \right\}\right|=o(|\cM_i|).
$$
\end{claim}
Before we prove Claim~\ref{claim:smallError}, we show how to complete the argument assuming this statement. Indeed, as the contribution from each term in $\cM_i$ is the same, Claim~\ref{claim:smallError} shows that the error introduced by replacing $\cM_i^*$ with $\cM_i$ in~\eqref{eq:domTerm2} is negligible: it is accounted for by a factor of $(1\pm o(1))$. Moreover, note that 
$$
\sum_{m_i}\sum_{(T,T')\in\cM_i}\frac{\EE\left(\IndicatorRVAPCentralised{T}\IndicatorRVAPCentralised{T'}\right)}{\sigma_{|T|}\sigma_{|T'|}}=
\begin{cases}
\EE(\RVnumberAPCentralised{\ell_1}{2})/\sigma_{\ell_1}^2=1,& \text{ for } i\in\cP_1,\\
\EE(\RVnumberAPCentralised{\ell_2}{2})/\sigma_{\ell_2}^2=1,& \text{ for } i\in\cP_2,\\
2\EE(\RVnumberAPCentralised{\ell_1}{}\RVnumberAPCentralised{\ell_2}{})/(\sigma_{\ell_1}\sigma_{\ell_2})\to 2\covarianceLeadingConstant{\ell_1}{\ell_2},& \text{ for } i\in\cP_3.
\end{cases}
$$
Consequently, we obtain
\begin{align*}
\mainContribution{k} &=(1\pm o(1))(k-1)!!\sum_{\theta_{1}=0}^{k/2}\sum_{\theta_{2}=0}^{k/2-\theta_{1}}\binom{k/2}{\theta_{1},\theta_{2},\theta_{3}}\left(u_{\ell_1}^2\right)^{\theta_{1}}\left(u_{\ell_2}^2\right)^{\theta_{2}}(2u_{\ell_1}u_{\ell_2}\covarianceLeadingConstant{\ell_1}{\ell_2})^{\theta_3}\\
&=(1 \pm o(1))(k-1)!!\left[u_{\ell_1}^2+u_{\ell_2}^2+2u_{\ell_1}u_{\ell_2}\covarianceLeadingConstant{\ell_1}{\ell_2}\right]^{k/2}, 
\end{align*}
as claimed.

\begin{proof}[Proof of Claim~\ref{claim:smallError}]
Fix an arbitrary $R\subseteq [n]$ of size at most $k\ell_1$. Denote $\cM_i':=\left\{(T,T')\in\cM_i \colon (T\cup T')\cap R\neq\emptyset \right\}$. Note first that once $T$ is fixed, the number of choices for $T'$ with $|T\cap T'|\ge 2$ is at most $O(\ell_1^4)$, as $T'$ is completely determined by choosing two elements in $T$ (for which there are at most $\ell_1^2$ choices) and deciding their positions within $T'$ (also at most $\ell_1^2$ choices).

We first deal with the case $m_i\ge 2$. We will see that $|\cM_i|=\Omega(n^2/\ell_1)$ and $|\cM_i'|=O(n\ell_1^5)$, and thus $|\cM_i'|=o(|\cM_i|)$ since $\ell_1=o(\log n)$. Indeed, note that for every $2\ell_1$-AP $T''$, we can let $T:=\{T''(1),\ldots,T''(|T|)\}$ and $T':=\{T''(|T|-m_i+1),\ldots,T''(|T|+|T'|-m_i)\}$. Then $(T,T')\in \cM_i$. Thus, $|\cM_i|\ge \numberAP{2\ell_1}=\frac{n^2}{2(2\ell_1-1)} (1 - o(1) )$ by Claim~\ref{obs:numberAP}. On the other hand, to obtain a pair $(T,T')$ in $\cM_i'$, we need to choose first some $x\in (T\cup T')\cap R$, which has at most $|R|\le k\ell_1$ choices. Then the arithmetic progression containing $x$, say $T$, is determined by picking a common difference, for which there are at most $n$ choices. Then by the observation above, the number of choices for $T'$ with $|T\cap T'|\ge 2$ is $O(\ell_1^4)$. Therefore, $|\cM_i'|=O(\ell_1\cdot n\cdot \ell_1^4)$ as claimed.

We then deal with the case $m_i=1$. Similarly, we show that $|\cM_i|=\Omega(n^3\ell_1^{-2})$ and $|\cM_i'|=O(n^2\ell_1)$, and hence $|\cM_i'|=o(|\cM_i|)$ since $\ell_1=o(\log n)$. Indeed, to get a pair $(T,T')$ in $\cM_i$, we have at least $\numberAP{\ell_1}$ choices to fix $T$ and then, upon choosing some $x\in T$ as its intersection with $T'$, there are at least $\frac{n/2}{\ell_1-1}$ choices to choose the common difference of $T'$. This is because if $x\ge n/2$ ($x\le n/2$ respectively), then we can find $T'$ with $x$ as the last (first respectively) element. Again there are at most $O(\ell_1^5)$ such $T'$ intersecting with $T$ at more than one place, we then have $|\cM_i|\ge \numberAP{\ell_1}\cdot \frac{n/2}{\ell_1-1}-O(\ell_1^5)=\Omega\left(n^3\ell_1^{-2}\right),$ by Claim~\ref{obs:numberAP}.
On the other hand, a pair in $\cM_i'$ is determined by choosing their single intersection point with $R$ and their common differences. So $|\cM_i'|\le |R|\cdot n\cdot n=O(n^2\ell_1)$, completing the proof of the claim.
\end{proof}

As demonstrated earlier, this also completes the proof of Lemma~\ref{lem:domTerm}.
\end{proof}

\subsection{Minor contribution to the moments}\label{Subsec:MinorContribution}

Next we turn our attention to $k$-tuples in $\minorContributionSet{k}=\left(\setAP{\ell_1}{}\cup\setAP{\ell_2}{}\right)^k\setminus \mainContributionSet{k}$, where $k\in\NN$ and $\mainContributionSet{k}=\emptyset$ if $k$ is odd.
\begin{lemma}\label{lem:minorTerm}
Let $k\in\NN$, we have
$$
\minorContribution{k}=\sum_{\mathbf{T}\in\minorContributionSet{k}}\EE\left(\prod_{i=1}^{k}\frac{u_{|T_{i}|}\IndicatorRVAPCentralised{T_{i}}}{\sigma_{|T_{i}|}}\right)=o(1). 
$$
\end{lemma}

We start with some preparation. We will change the order of summation in an algorithmic fashion as described below. First we fix an arbitrary total order $\totalOrder$ of the set $\setAP{\ell_1}{}\cup\setAP{\ell_2}{}$ such that all $\ell_1$-APs come before any $\ell_2$-AP, i.e.\ we have $\totalOrder(T)<\totalOrder(T')$ for all $T\in \setAP{\ell_1}{}$ and $T'\in\setAP{\ell_2}{}$. We now explore any (non-empty) finite collection  of APs component by component. Roughly speaking, given $\mathbf{T}$, let $H$ be an auxiliary $k$-vertex graph, in which each vertex represents an AP in $\mathbf{T}$ and two vertices are adjacent if and only if the corresponding APs have non-empty intersection. Then we will explore $V(H)$, moving from one vertex to one of its neighbours according to the ordering $\totalOrder$ and start the search from a new component whenever the current one is exhausted. For $ \mathbf{T} \in \bigcup_{k\in\NN}\left(\setAP{\ell_1}{}\cup\setAP{\ell_2}{}\right)^k $, we set $ |\mathbf{T}| := \inf\left\{k \geq 1 : \mathbf{T} \in \left(\setAP{\ell_1}{}\cup\setAP{\ell_2}{}\right)^k \right\} $. More precisely, we perform the following algorithm:

\medskip

\noindent\rule{3cm}{3pt}

\begin{enumerate}[(I)]
\item[] \hspace{-1cm}\textsc{INPUT}: $\mathbf{T}\in \bigcup_{k\in\NN}\left(\setAP{\ell_1}{}\cup\setAP{\ell_2}{}\right)^k$.\vspace{.2cm}
\item Initialise the inactive list $\cL_i$ and active list $\cL_a$: $\cL_i\leftarrow\mathbf{T}$, $\cL_a\leftarrow\emptyset$, and $j\leftarrow1$.
\item Start a new component: If $\cL_a=\emptyset$, then let $\cL_a\leftarrow \{\min_{\totalOrder}\cL_i\}$.
\item Set:
\begin{align*}
T_j&\leftarrow \min_{\totalOrder}\cL_a,\\ 
\typeSize_j&\leftarrow |T_j|,  \quad (\mbox{size of the current AP})\\
\type_j&\leftarrow \left|T_j\cap \bigcup_{j'=1}^{j-1}T_{j'}\right|,\quad (\mbox{size of the overlap with previous APs})\\
\cC &\leftarrow \{T\in \cL_i\colon T\cap T_j\neq\emptyset\}. \quad (\mbox{current component})
\end{align*}
\item Update: 
\begin{align*}
\cL_a&\leftarrow (\cL_a\cup \cC)\setminus\{T_j\},\\
\cL_i&\leftarrow \cL_i\setminus \cC.
 \end{align*}
\item If $j=|\mathbf{T}|$, then STOP; otherwise, set $j\leftarrow j+1$ and return to step (II).\vspace{.2cm}
\item[] \hspace{-1cm}\textsc{OUTPUT}: $\reordering(\mathbf{T}):=(T_1,\dots, T_{|\mathbf{T}|})$ and $\typeFunction(\mathbf{T}):=(\typeVect,\typeSizeVect)$, where $\typeVect=(\type_1,\type_2,\ldots, \type_{|\mathbf{T}|})$ and $\typeSizeVect=(\typeSize_1,\typeSize_2,\ldots, \typeSize_{|\mathbf{T}|})$. 
\end{enumerate}

\noindent\rule{3cm}{3pt}

\medskip

Note that  any permutation $\mathbf{T}'$ of the input $\mathbf{T}$ will result in the same ordered tuple $\reordering(\mathbf{T}')=(T_1,\dots,T_{|\mathbf{T}|})$. We now assume that $|\mathbf{T}|=k$. Observe that $\typeVect$ and $\typeSizeVect$ satisfy 
\begin{align}\label{eq:valid0}
\forall i\in[k]&\colon \typeSize_i\in\{\ell_1,\ell_2\},\\\label{eq:valid1}
\forall i\in[k]&\colon 0\le \type_i\le \typeSize_i,\\
\forall i\in[k]&\colon \{\type_i=0\}\implies \{\typeSize_i=\ell_1\} \vee\{\typeSize_j=\ell_2,\, \forall j=i,\dots,k\}, \label{eq:valid2}
\end{align}
where~\eqref{eq:valid2} follows from the choice of $\totalOrder$.

For $r\in\{0,1,\dots,\ell_1\}$ and $j\in\{1,2\}$ we define index sets 
\begin{align*}
\IndexSetMixed{r}{\ell_j}:=\{i\in[k]\colon \type_i=r , \ \typeSize_i=\ell_j\}, \qquad \IndexSet{r}:=\IndexSetMixed{r}{\ell_1}\cup\IndexSetMixed{r}{\ell_2}.
\end{align*}
Additionally, note that if the input $\mathbf{T}$ is such that there exists $ i\in[k-1]$ for which $ \type_i=\type_{i+1}=0$, then $T_i$ is disjoint from $\bigcup_{j\in[k]\setminus\{i\}}T_j$ implying $\EE\left(\prod_{j=1}^{k}\IndicatorRVAPCentralised{T_{j}}\right)=\EE\IndicatorRVAPCentralised{T_i}\cdot\EE\left(\prod_{j\in [k]\setminus\{i\}}\IndicatorRVAPCentralised{T_{j}}\right)=0$, i.e.~such $\mathbf{T}$ does not contribute to $\minorContribution{k}$. Consequently, we have 
\begin{equation}\label{eq:consecZero}
\forall i\in[k-1]\colon \type_i+\type_{i+1}>0.
\end{equation}
Similarly, $\type_k>0$, since otherwise $\IndicatorRVAPCentralised{T_k}$ is independent from $ (\IndicatorRVAPCentralised{T_1}, \dots, \IndicatorRVAPCentralised{T_{k - 1}}) $ and thus $\mathbf{T}$ does not contribute to $\minorContribution{k}$. We write $$\typeVectSpaceValid{k}:=\{\typeVect\in\{0,1,\ldots,\ell_1\}^k\colon \typeVect \text{ satisfies }\eqref{eq:consecZero}\text{ and } \type_k>0\}$$ 
for the set of all \emph{type} vectors of length $k$ which do not contain two consecutive zeros and do not end in a zero. In particular, this implies that we may assume $\IndexSetSize{0}\le \frac{k}{2}-1$ for even $k$ and $\IndexSetSize{0}\le \frac{k-1}{2}$ for odd $k$, in other words, we have \begin{equation}\label{eq:maxComponents}
\IndexSetSizeMixed{0}{\ell_1}+\IndexSetSizeMixed{0}{\ell_2}\le \left\lceil k/2 \right\rceil-1.
\end{equation}

Next, for any type vector $\typeVect\in\{0,1,\dots,\ell_1\}^k$, we define the set of \emph{valid size-type} vectors
$$
\typeSizeVectSpace{k}{\typeVect}:=\left\{\typeSizeVect\in \{\ell_1,\ell_2\}^k\colon (\typeVect,\typeSizeVect) \text{ satisfies }\eqref{eq:valid0},~\eqref{eq:valid1}, \text{ and}~\eqref{eq:valid2} \right\}.
$$

The main idea is to enumerate the sum in~\eqref{eq:kthMoment} by first choosing the vector $\typeVect\in\{0,1,\dots,\ell_1\}^k$, then a valid size-type vector $\typeSizeVect\in\typeSizeVectSpace{k}{\typeVect}$, and lastly a tuple $(T_1,\dots,T_k)\in \minorContributionSet{k}$ such that $\typeFunction(T_1,\dots,T_k)=(\typeVect,\typeSizeVect)$. In terms of formula, we obtain
\begin{align}\label{Eq:SummationWithBFS}
\minorContribution{k}=\sum_{\typeVect\in\typeVectSpaceValid{k}}\sum_{\typeSizeVect\in\typeSizeVectSpace{k}{\typeVect}}\sum_{\substack{\mathbf{T}\in \minorContributionSet{k}\\\typeFunction(\mathbf{T})=(\typeVect,\typeSizeVect)} }\EE\bigg(\prod_{T\in \mathbf{T}}\frac{u_{|T|}\IndicatorRVAPCentralised{T}}{\sigma_{|T|}}\bigg)=\sum_{\typeVect\in\typeVectSpaceValid{k}}\sum_{\typeSizeVect\in\typeSizeVectSpace{k}{\typeVect}}\numberMixedAPTuples{\typeVect,\typeSizeVect}\cdot \contribution{\typeVect,\typeSizeVect},
\end{align}
where
\begin{align*}
\numberMixedAPTuples{\typeVect,\typeSizeVect}:=\left|\left\{\mathbf{T}\in \minorContributionSet{k}\colon \typeFunction(\mathbf{T})=(\typeVect,\typeSizeVect)\right\}\right|
\end{align*}
denotes the number of tuples with given type vectors $(\typeVect,\typeSizeVect)$, and
\begin{align*}
\contribution{\typeVect,\typeSizeVect}:= \left(\prod_{i\in[k]}\frac{u_{\typeSize_i}}{\sigma_{\typeSize_i}}\right) \frac{1}{\numberMixedAPTuples{\typeVect,\typeSizeVect} } \sum_{\substack{\mathbf{T}\in \minorContributionSet{k}\\\typeFunction(\mathbf{T})=(\typeVect,\typeSizeVect)}}\EE(\IndicatorRVAPCentralised{T_1}\cdots \IndicatorRVAPCentralised{T_k})
\end{align*}
is the average contribution to $\minorContribution{k}$ of a $k$-tuple with given type vectors $(\typeVect,\typeSizeVect)$.

We first aim to bound the average contribution $ \contribution{\typeVect,\typeSizeVect} $.

\begin{prop}\label{prop-ma}
Let $\typeVect\in \typeVectSpaceValid{k}$ and $\typeSizeVect\in\typeSizeVectSpace{k}{\typeVect}$, then we have 
$$
\contribution{\typeVect,\typeSizeVect}=(1\pm o(1))\left(\prod_{i\in[k]}\frac{u_{\typeSize_i}}{\sigma_{\typeSize_i}}\right)p^{\sum_{i\in[k]}(\typeSize_i-\type_i)}.
$$
\end{prop}

\begin{proof}
Let $\mathbf{T}=(T_1,\dots,T_k)\in\minorContributionSet{k}$ with $\typeFunction(\mathbf{T})=(\typeVect,\typeSizeVect)$. Here $T_1,\ldots,T_k$ are in the order corresponding to the output of the exploring algorithm, hence we have $ |T_i| = \typeSize_i $. We see that
\begin{align*}
\EE(\IndicatorRVAPCentralised{T_1}\cdots \IndicatorRVAPCentralised{T_k}) = \EE\!\left(\prod_{i\in[k]}\left(\Ind{\{T_i\subseteq [n]_p\}}-p^{|T_i|}  \right) \!\!\right) \!\! = \sum_{ R \subseteq [k] }\!\!\left(\prod_{i\in R}-p^{\typeSize_i} \! \right) \PP\! \left(\bigcup_{i\not\in R}T_i\subseteq [n]_p\right).
\end{align*}
First observe that the summand $Q_\emptyset$ for $R = \emptyset$ is given by
$$
Q_\emptyset:=\PP\left(\bigcup_{i\in[k]}T_i\subseteq [n]_p\right)=p^{\sum_{i=1}^{k}(\typeSize_i-\type_i)},
$$ 
as $\typeSize_i-\type_i=|T_i\setminus \cup_{i'\le i-1}T_{i'}|$ and so $\sum_{i\in[k]}(\typeSize_i-\type_i)=|\cup_{i\in[k]}T_i|$. Thus, it only remains to show that the remaining (constantly many) summands are all of lower order. 

Let $r \in [k] $ and fix an arbitrary subset $R\subseteq [k]$ of size $r$. The absolute value of its contribution to $\EE(\IndicatorRVAPCentralised{T_1}\cdots \IndicatorRVAPCentralised{T_k})$ is equal to
\begin{align*}
Q_R := \left(\prod_{i\in R}p^{\typeSize_i}\right)\,\PP\left(\bigcup_{i\not\in R}T_i\subseteq [n]_p\right)= p^{\sum_{i\in R}\typeSize_i+|\cup_{i\notin R}T_i|}.
\end{align*}
Note that 
$$
\sum_{i\in R}\typeSize_i+|\cup_{i\notin R}T_i|=\sum_{i\in R}|T_i|+|\cup_{i\notin R}T_i|\ge |\cup_{i\in[k]}T_i|=\sum_{i\in[k]}(\typeSize_i-\type_i).
$$ 
Furthermore, if this last inequality is not an equality, then 
$$ 
Q_R \le  p^{\sum_{i=1}^{k}(\typeSize_i-\type_i)+1}=o\left(p^{\sum_{i=1}^{k}(\typeSize_i-\type_i)}\right),
$$
i.e.\ $Q_R$ is negligible compared to $Q_\emptyset$. 

Next, suppose towards contradiction that the equality holds, so 
$$
\sum_{i\in R}|T_i|=|\cup_{i\in[k]}T_i|-|\cup_{i\notin R}T_i|. 
$$
But at the same time we have
$$
\sum_{i\in R}|T_i|\ge |\cup_{i\in R}T_i|\ge |\cup_{i\in[k]}T_i|-|\cup_{i\notin R}T_i|,
$$
and thus all intermediate inequalities above must be equalities. This happens for the first inequality when $\{T_i\}_{i\in R}$ are pairwise disjoint and for the second inequality when $\left(\cup_{i\in R}T_i\right)\cap\left(\cup_{i\notin R}T_i\right)=\emptyset$. But this in turn implies that for any $i\in R$, the set $T_i$ is disjoint from $\cup_{j\neq i}T_j$, so $\type_i=\type_{i+1}=0$, contradicting~\eqref{eq:consecZero}.

Because these bounds are uniform over the choice of the $k$-tuple $\mathbf{T}$ the statement follows by taking the average.
\end{proof}

We now aim at bounding the number of summands $\numberMixedAPTuples{\typeVect,\typeSizeVect}$. To do so, recall that in the dependency graph $\depGraph{\ell_1,\ell_2}$ defined in~\eqref{eq:depGraph}, each vertex represents an AP in $\setAP{\ell_1}{}\cup\setAP{\ell_2}{}$, and two vertices form an edge if and only if the corresponding APs have non-empty intersection.
\begin{prop}\label{prop-sa}
For all $\typeVect\in\typeVectSpaceValid{k}$ and $\typeSizeVect\in\typeSizeVectSpace{k}{\typeVect}$, we have 
\begin{align*}
\numberMixedAPTuples{\typeVect,\typeSizeVect}&=O(1)\cdot \prod_{j=1}^{\IndexSetSize{0}}\Bigg[\left(\frac{n^2}{\typeSize_{r_j}}\right)\cdot(n\typeSize_{r_j+1})^{\Ind{\{\type_{r_j+1}=1\}}}(\typeSize_{r_j}^2\typeSize_{r_j+1}^2)^{\Ind{\{\type_{r_j+1}\ge 2\}}-\Ind{\{\type_{r_j+1}=\typeSize_{r_j+1}=\typeSize_{r_j}\}}}\\
&\hspace{6.8cm}\cdot \prod_{i=r_j+2}^{r_{j+1}-1}(n \ell_1)^{\Ind{\{\type_{i}=1\}}}(\typeSize_{i}^2\ell_1^2)^{\Ind{\{\type_{i}\ge 2\}}}\Bigg].
\end{align*}
\end{prop}

\begin{proof}
First, note that for any $\mathbf{T}$ such that $\typeFunction(\mathbf{T})=(\typeVect,\typeSizeVect)$ and $\reordering(\mathbf{T})=(T_1,\dots,T_k)$, the component structure of the induced subgraph $\depGraph{\ell_1,\ell_2}\!\!\left[\bigcup_{i\in[k]}T_i\right]$ is already determined by the type-vector  $\typeVect$. More precisely, for $j=1,\dots, \IndexSetSize{0}$, let 
$$
r_j:=\min\left\{i\in[k]\setminus\{r_1,\dots,r_{j-1}\}\colon \type_i=0\right\}
$$
denote the $j$-th  zero entry of $\typeVect$, and set $r_{\IndexSetSize{0}+1}:=k+1$. Note that $r_1=1$ and $\{T_{r_j},\dots, T_{r_{j+1}-1}\}$ forms a component of $\depGraph{\ell_1,\ell_2}\!\!\left[\bigcup_{i\in[k]}T_i\right]$ for all $j=1,\dots, \IndexSetSize{0}$.

We will construct tuples $\mathbf{T}$ with $\typeFunction(\mathbf{T})=(\typeVect,\typeSizeVect)$ in the order given by its reordering $\reordering(\mathbf{T})=(T_1,\dots,T_k)$. In particular, this means that we consider one component of $\depGraph{\ell_1,\ell_2}\!\!\left[\bigcup_{i\in[k]}T_i\right]$ after the other. Let $j=1,\dots, \IndexSetSize{0}$ and assume that $T_1,\dots, T_{r_j-1}$ have already been chosen.

Observe that, by~\eqref{eq:consecZero}, the $j$-th component contains at least two APs $T_{r_j}$ and $T_{r_j+1}$.  As  $T_{r_j}$ starts a new component ($\type_{r_j}=0$), the number of choices for $T_{r_j}$ is at most $\numberAP{\typeSize_{r_j}}=O(n^2\typeSize_{r_j}^{-1})$ by Claim~\ref{obs:numberAP}. Next we choose $T_{r_j+1}$: 
\begin{enumerate}[(a)]
\item if $\type_{r_j+1}=1$, then the number of choices is at most $O(n\typeSize_{r_j})$, since there are at most $\typeSize_{r_j}$ choices for the common vertex $x\in T_{r_j}\cap T_{r_j+1}$, at most $\typeSize_{r_j+1}$ choices for the position of  $x$ within $T_{r_j+1}$ and $O(n/ \typeSize_{r_j+1})$ for the common difference of $T_{r_j+1}$;
\item if $\type_{r_j+1}=\typeSize_{r_j+1}=\typeSize_{r_j}$, then there is only one possibility $T_{r_j+1}=T_{r_j}$;
\item otherwise, $T_{r_j+1}$ is  determined by choosing two elements from $T_{r_j}$ and their respective positions within $T_{r_j+1}$, which amounts to at most $O(\typeSize_{r_j}^2\typeSize_{r_j+1}^2)$ many choices. 
\end{enumerate}
Similarly, for any remaining $i=r_j+2,\dots, r_{j+1}-1$ (there might be none), we use the following bounds on the number of choices for $T_{i}$:
\begin{enumerate}[(a)]
\item if $\type_{r_j+1}=1$, then the number of choices is at most $O(n\ell_1)$, since there are at most $O(\ell_1)$ choices for the common vertex $x\in T_{i}\cap \left( T_{r_j}\cup\dots\cup T_ {i-1}\right)$, at most $\typeSize_{i}$ choices for the position of  $x$ within $T_{i}$ and $O(n/ \typeSize_{i})$ for the common difference of $T_{i}$;
\item otherwise, $T_{i}$ is  determined by choosing two elements from $T_{r_j}\cup \dots\cup T_{i-1}$ and their respective positions within $T_{i}$, which amounts to at most $O(\ell_1^2\typeSize_{i}^2)$ many choices.
\end{enumerate}
The claim follows by multiplying for all $j=1,\dots,\IndexSetSize{0}$ and $i=r_j,\dots, r_{j+1}-1$.
\end{proof}

With this preparation we are now ready to prove Lemma~\ref{lem:minorTerm}. We will bound the contribution of each $k$-tuple to $\minorContribution{k}=\sum_{\typeVect,\typeSizeVect}\contribution{\typeVect,\typeSizeVect}\numberAPTuples{\typeVect,\typeSizeVect}$ from above component-wise 
\begin{proof}[Proof of Lemma~\ref{lem:minorTerm}]
First observe that Lemma~\ref{lem:covariance} implies that for any $\ell\ge 3$ we have 
\begin{equation}\label{eq:sigmaInequalitySTP}
\sigma_{\ell}^{-1}=O(\numberAP{\ell} p^{\ell})^{-1/2}\stackrel{C.\ref{obs:numberAP}}{=}O(n^{-1}p^{-\ell/2}\ell^{1/2}),
\end{equation}
and also
\begin{equation}\label{eq:sigmaInequalityLP}
\sigma_{\ell}^{-1}=O(\numberLoosePairs{\ell} p^{2\ell-1})^{-1/2}\stackrel{L.\ref{lem:mix-loose-pair}}{=}O(n^{-3/2}p^{-\ell+1/2}).
\end{equation}

Using Propositions~\ref{prop-ma} and~\ref{prop-sa} the expression in~\eqref{Eq:SummationWithBFS} becomes
\begin{equation}\label{eq:minorContribution}
\minorContribution{k}=O(1)\cdot \sum_{\typeVect\in\typeVectSpaceValid{k}}\sum_{\typeSizeVect\in\typeSizeVectSpace{k}{\typeVect}} \prod_{i\in[k]} g_{\typeVect,\typeSizeVect}(i)\sigma_{\typeSize_{i}}^{-1},
\end{equation}
where 
$$
g_{\typeVect,\typeSizeVect}(i):=
\begin{cases}
n^2\typeSize_i^{-1}p^{\typeSize_i}&\text{ if } \type_i=0;\\
n\typeSize_i p^{\typeSize_i-1}&\text{ if } \type_i=1\wedge \type_{i-1}=0;\\
n\ell_1 p^{\typeSize_i-1} &\text{ if } \type_i=1\wedge \type_{i-1}>0;\\
1 &\text{ if } \type_i=\typeSize_i=\typeSize_{i-1}\wedge \type_{i-1}=0;\\
\typeSize_i^2\typeSize_{i-1}^2 p^{\typeSize_i-\type_i} &\text{ if } 2\le \type_i\le \typeSize_i-1\wedge \type_{i-1}=0;\\
\typeSize_i^2\typeSize_{i-1}^2  &\text{ if } \type_i=\typeSize_i\wedge \typeSize_i\neq\typeSize_{i-1}\wedge \type_{i-1}=0;\\
\typeSize_i^2\ell_1^2 p^{\typeSize_i-\type_i} &\text{ if } \type_i\ge 2\wedge \type_{i-1}>0.
\end{cases}
$$
Moreover, we recall the notation $r_j=\min\left\{i\in[k]\setminus\{r_1,\dots,r_{j-1}\}\colon \type_i=0\right\}$ and $r_{\IndexSetSize{0}+1}:=k+1$ used in the proof of Proposition~\ref{prop-sa}. These indices split the interval $[k]$ into $\IndexSetSize{0}$ parts, i.e.\  $[k]=\dot{\bigcup}_{j=1}^{\IndexSetSize{0}}\{r_j,\dots,r_{j+1}-1\}$, where each part has size at least two  as, by~\eqref{eq:consecZero}, $\typeVect$ does not have consecutive zeros. Now, fix any $j\in[\IndexSetSize{0}]$. We first bound the first two factors together:
\begin{enumerate}[(a)]
\item If $\type_{r_j+1}=1$, then we have 
\begin{equation*}
\frac{g_{\typeVect,\typeSizeVect}(r_j)g_{\typeVect,\typeSizeVect}(r_j+1)}{\sigma_{\typeSize_{r_j}}\sigma_{\typeSize_{r_j+1}}} = n^2\typeSize_{r_{j}}^{-1}p^{\typeSize_{r_{j}}}\cdot n\typeSize_{r_j+1} p^{\typeSize_{r_j+1}-1}\cdot \sigma_{\typeSize_{r_j}}^{-1}\sigma_{\typeSize_{r_j+1}}^{-1}\stackrel{\eqref{eq:sigmaInequalityLP}}{=}O(1),
\end{equation*}
because the validity condition~\eqref{eq:valid2} implies $\typeSize_{r_j+1}\le\typeSize_{r_j}$, since $\type_{r_j}=0$ .

\item If $2\le \type_{r_j+1}\le \typeSize_{r_j+1}-1$, then we have
\begin{align*}
\frac{g_{\typeVect,\typeSizeVect}(r_j)g_{\typeVect,\typeSizeVect}(r_j+1)}{\sigma_{\typeSize_{r_j}}\sigma_{\typeSize_{r_j+1}}} &=n^2\typeSize_{r_{j}}^{-1}p^{\typeSize_{r_{j}}}\cdot \typeSize_{r_j}^2\typeSize_{r_j+1}^2 p^{\typeSize_{r_j+1}-\type_{r_j+1}}\cdot \sigma_{\typeSize_{r_j}}^{-1}\sigma_{\typeSize_{r_j+1}}^{-1}\\
&\stackrel{\eqref{eq:sigmaInequalitySTP}}{=}O\left(p^{1+(\typeSize_{r_j}-\typeSize_{r_j+1})/2}\ell_1^{4}\right)=o(\ell_1^{-1})
\end{align*}
since $p\ell_1^5\to 0$, by assumption~\eqref{eq:assumptionPUpper}.

\item If $\type_{r_j+1}=\typeSize_{r_j+1}\le\typeSize_{r_j}-1$, then we have
\begin{align*}
\frac{g_{\typeVect,\typeSizeVect}(r_j)g_{\typeVect,\typeSizeVect}(r_j+1)}{\sigma_{\typeSize_{r_j}}\sigma_{\typeSize_{r_j+1}}} &=n^2\typeSize_{r_{j}}^{-1}p^{\typeSize_{r_{j}}}\cdot \typeSize_{r_j}^2\typeSize_{r_j+1}^2\cdot \sigma_{\typeSize_{r_j}}^{-1}\sigma_{\typeSize_{r_j+1}}^{-1}\\
&\stackrel{\eqref{eq:sigmaInequalitySTP}}{=}O\left(p^{(\typeSize_{r_j}-\typeSize_{r_j+1})/2}\ell_1^{4}\right)=o(1),
\end{align*}
since $p^{1/2}\ell_1^4\to 0$, by assumption~\eqref{eq:assumptionPUpper}.

\item If $\type_{r_j+1}=\typeSize_{r_j+1}=\typeSize_{r_j}$, then we have
\begin{align*}
\frac{g_{\typeVect,\typeSizeVect}(r_j)g_{\typeVect,\typeSizeVect}(r_j+1)}{\sigma_{\typeSize_{r_j}}\sigma_{\typeSize_{r_j+1}}} &=n^2\typeSize_{r_{j}}^{-1}p^{\typeSize_{r_{j}}}\cdot \sigma_{\typeSize_{r_j}}^{-1}\sigma_{\typeSize_{r_j+1}}^{-1}\stackrel{\eqref{eq:sigmaInequalitySTP}}{=}O(1).
\end{align*}
\end{enumerate}

We now treat any (potentially) remaining indices $i=r_j+2,\dots, r_{j+1}-1$ and estimate $g_{\typeVect,\typeSizeVect}(i)\sigma_{\typeSize_i}^{-1}$ one by one.
\begin{enumerate}[(a)]

\item If $\type_i=1$,  then we have
\begin{align*}
g_{\typeVect,\typeSizeVect}(i)\sigma_{\typeSize_i}^{-1}= n \ell_1p^{\typeSize_i-1}\sigma_{\typeSize_i}^{-1} \stackrel{\eqref{eq:sigmaInequalityLP}}{=} O((np)^{-1/2}\ell_1)=o(1),
\end{align*}
since $3\le \ell_1=o(\log n)$ and $np=\Omega(n^{1-2/\ell_1})=n^{\Omega(1)}$, by assumption~\eqref{eq:assumptionPLower}.

\item If $2\le \type_i\le\typeSize_i-1$, then we have
\begin{align*}
g_{\typeVect,\typeSizeVect}(i)\sigma_{\typeSize_i}^{-1}\le  \typeSize_{i}^2\ell_1^2p^{\typeSize_i-\type_i}\sigma_{\typeSize_i}^{-1}
=O(p\ell_1^4)=o(\ell_1^{-1}), 
\end{align*}
since $\sigma_{\typeSize_{i}}^{-1}=O(1)$ by assumption~\eqref{eq:assumptionPLower} and because $p\ell_1^5\to 0$, by assumption~\eqref{eq:assumptionPUpper}.

\item However, if $\type_i=\typeSize_i$, then we have
\begin{align*}
g_{\typeVect,\typeSizeVect}(i)\sigma_{\typeSize_i}^{-1}\le \typeSize_{i}^2\ell_1^2\sigma_{\typeSize_i}^{-1}=O(\sigma_{\ell_1}^{-1}\ell_1^4)\stackrel{\eqref{eq:sigmaInequalitySTP}}{=}O(n^{-1}p^{-\ell_1/2}\ell_1^{9/2})=o(1), 
\end{align*}
since $n^2p^{\ell_1}\ell_1^{-9}\to+\infty$, by assumption~\eqref{eq:assumptionPLower}.
\end{enumerate}

Next, we observe that by~\eqref{eq:maxComponents}, we have $\IndexSetSize{0}\le \lceil k/2\rceil-1$, implying that there must be at least one $j\in \IndexSetSize{0}$ such that there exists an integer $i_0$ satisfying $r_j+2\le i_0\le r_{j+1}-1$. But then, the previous computation shows that the corresponding factor $g_{\typeVect,\typeSizeVect}(i_0)\sigma_{\typeSize_{i_0}}^{-1}$ must be small. More precisely, we have 
$$
g_{\typeVect,\typeSizeVect}(i_0)\sigma_{\typeSize_{i_0}}^{-1}=o(1)\cdot \ell_1^{-\Ind{\{2\le \type_{i_0}\le\typeSize_{i_0}-1\}}}.
$$
Consequently, from~\eqref{eq:minorContribution} and multiplying the bounds for all $i\in[k]$ we obtain 
$$
\minorContribution{k}=O(1)\cdot \sum_{\typeVect\in\typeVectSpaceValid{k}}\sum_{\typeSizeVect\in\typeSizeVectSpace{k}{\typeVect}} o(\ell_1^{-Q(\typeVect,\typeSizeVect)}),
$$
where
$$
Q(\typeVect,\typeSizeVect):=\left|\left\{i\in[k]\colon 2\le \type_i\le \typeSize_i-1\right\}\right|.
$$
Last, for any $q\in\{0,1,\dots, k\}$, the number of summands with $Q(\cdot)=q$ is at most $\ell_1^q3^{k-q}2^k=O(\ell_1^q)$ yielding 
$$
\minorContribution{k}=o(1),
$$
thereby completing the proof of Lemma~\ref{lem:minorTerm}.
\end{proof}

\subsection{Completing the argument: application of the method of moments}

It remains to apply the method of moments to show the convergence to a (bivariate) Gaussian distribution.

\begin{proof}[Proof of Theorem~\ref{thm:mainBivariate}]
Recall from~\eqref{eq:split} that we have 
$$
\EE \left[\left( u_{\ell_1}\frac{\RVnumberAPCentralised{\ell_1}{}}{\sigma_{\ell_1}}+u_{\ell_2}\frac{\RVnumberAPCentralised{\ell_2}{}}{\sigma_{\ell_2}}\right)^k\right]   = \mainContribution{k} + \minorContribution{k},
$$
for all $k\in\NN$ and $u_{\ell_1},u_{\ell_2} \in \RR$. Furthermore, we have computed the asymptotics for $\mainContribution{k}$ and $\minorContribution{k}$ in Lemmas~\ref{lem:domTerm} and~\ref{lem:minorTerm}, implying that for even $k$ we have
$$
\EE \left[\left( u_{\ell_1}\frac{\RVnumberAPCentralised{\ell_1}{}}{\sigma_{\ell_1}}+u_{\ell_2}\frac{\RVnumberAPCentralised{\ell_2}{}}{\sigma_{\ell_2}}\right)^k\right] =(1\pm o(1))(k-1)!!\left[u_{\ell_1}^2+u_{\ell_2}^2+2\covarianceLeadingConstant{\ell_1}{\ell_2}u_{\ell_1}u_{\ell_2}\right]^{k/2},
$$
and for odd $k$ we have 
$$
\EE \left[\left( u_{\ell_1}\frac{\RVnumberAPCentralised{\ell_1}{}}{\sigma_{\ell_1}}+u_{\ell_2}\frac{\RVnumberAPCentralised{\ell_2}{}}{\sigma_{\ell_2}}\right)^k\right] =o(1).
$$
Letting $n\to+\infty$ we obtain the $k$-th moments of the bivariate standard Gaussian distribution with covariance $\covarianceLeadingConstant{\ell_1}{\ell_2}$. Hence, Theorems~\ref{thm:MethodOfMoments} and~\ref{thm:CramerWoldDevice} imply that 
$$
\left(\frac{\RVnumberAPCentralised{\ell_1}{}}{\sigma_{\ell_1}},\frac{\RVnumberAPCentralised{\ell_2}{}}{\sigma_{\ell_2}}\right)\stackrel{d}{\tendsto{n}{+\infty}}\Gaussian{0}{0}{1}{\covarianceLeadingConstant{\ell_1}{\ell_2}}{1}.
$$
The distinction of the different regimes in Theorem~\ref{thm:mainBivariate} follows from Lemma~\ref{lem:kappa}, completing the proof.
\end{proof}

The same proof also applies for the study of univariate fluctuations.\footnote{Albeit with the mild additional assumption $p\ell^{9}\to 0$ for technical reasons.}
\begin{proof}[Alternative proof of Theorem~\ref{thm:mainUnivariate}\eqref{thm:mainUnivariateNormal}]
For $3\le \ell=\ell(n)=o(\log n)$ and $0<p=p(n)<1$ such that $p\ell^{9}\to 0$ and $n^2p^{\ell}\ell^{-9}\to +\infty$, we obtain
$$
\EE\left[\left(\frac{\RVnumberAPCentralised{\ell}{}}{\sigma_{\ell}}\right)^k\right]=
\begin{cases}
(1\pm o(1))(k-1)!! & \text{ for } k \text{ even},\\
o(1)&\text{ if }\text{ for } k \text{ odd},
\end{cases}
$$
from Lemmas~\ref{lem:domTerm} and~\ref{lem:minorTerm} by setting $\ell_2=\ell$, $\ell_1=2\ell$, $u_{\ell_2}=1$, and $u_{\ell_1}=0$. Letting $n\to+\infty$ Theorem~\ref{thm:MethodOfMoments} shows that 
$$
\frac{\RVnumberAPCentralised{\ell}{}}{\sigma_{\ell}}\stackrel{d}{\tendsto{n}{+\infty}}\No(0,1),
$$
as claimed.
\end{proof}

\section{Concluding remarks}\label{Sec:Concluding}
The main topic at stake in this article was to study the joint distribution of the numbers of APs of different length in some random subsets $M$ of the integers. In the most general setup, we would like to understand the growth behaviour of the family $\{\RVnumberAP{\ell}{}\}_{3\le \ell\le n}$ where $\RVnumberAP{\ell}{}=\RVnumberAP{\ell}{}(M)$ denotes the number of $\ell$-APs of integers which are (entirely) contained in $M$. Here, we took a first step in this direction by determining the joint limiting distribution of $(\RVnumberAP{\ell_1}{},\RVnumberAP{\ell_2}{})$ in $M=[n]_p$ for a significant range of parameters $p$ and $3\le \ell_2<\ell_1=o(\log n)$. We believe that our approach should also allow us to determine the limiting distribution of $r$-tuples $(\RVnumberAP{\ell_1}{}, \RVnumberAP{\ell_2}{} ,\dots,\RVnumberAP{\ell_r}{})$ for $r\ge 3$ (within the intersection of their respective Gaussian regimes), hence, to give a functional Central Limit Theorem for e.g. $\big(\RVnumberAP{\lfloor s \ell \rfloor }{} \big)_{s \in [0, 1]} $ with $ \ell=\ell(n) = o(\log n) $. In particular, it would be interesting to know whether for some constants $\ell_1,\ell_2,\dots,\ell_r$, with (constant) $r\ge 3$, the Gaussian limit becomes degenerate. We observed it for $r=2$ when $\ell_1=\ell_1(n), \ell_2=\ell_2(n)\to+\infty$ sufficiently slowly: $\RVnumberAP{\ell_1}{}$ and $\RVnumberAP{\ell_2}{}$ are then either asymptotically uncorrelated or converge to the same Gaussian random variable (after re-normalisation).

Furthermore, recall that Theorem~\ref{thm:mainBivariate} uses the assumption $n^2p^{\ell_1}\ell_1^{-9}\to+\infty$ which guarantees that both $\RVnumberAP{\ell_1}{}$ and $\RVnumberAP{\ell_2}{}$ are within their respective Gaussian regimes. One may thus ask what happens for smaller values of $p$. At least heuristically, our results for the overlap pair regime (i.e.\ $np^{\ell_1-1}\ell_1\to 0$) suggest that a good candidate for the joint limit consists of two independent random variables having the appropriate marginal distributions (Gaussian or Poisson) determined in Theorem~\ref{thm:mainUnivariate}. 

Throughout the article, we focused on $\ell$-APs where $\ell=o(\log n)$, the reason being that typically the random set $[n]_p$ will not contain any longer APs as long as $p=o(1)$. In order to witness any $\ell$-APs with $\ell/\log n\to+\infty$ we would need to consider $p=p(n)\to 1$. Borrowing some intuition from Gao and Sato's work~\cite{GaoSato} on large matchings in the random graph $G(n,p)$ -- namely the log-normal paradigm of Gao~\cite{Gao} -- we might expect to see another change of regime to a Log-normal limiting distribution for very long APs. However, in this regime, various estimates derived in this paper cease to hold and we leave this as an open problem.

Another question of interest concerns the behaviour of the joint cumulants of $(\RVnumberAP{\ell_1}{},\RVnumberAP{\ell_2}{})$ in the various regimes encountered here. In the Gaussian regime, since the moments of the rescaled random variables converge to the Gaussian moments, their cumulants of order $ r \geq 3 $ converge to $0$. One can ask if the BFS coding allows to see such a behaviour in a fine way, for instance with an asymptotic expansion. 

Lastly, we would like to move in a slightly different direction: let $0<s<t$ and 
consider the coupling $\left[\lfloor tn\rfloor\right]_p=\left[\lfloor sn\rfloor\right]_p\cup\left\{\lfloor sn\rfloor+1,\dots,\lfloor tn\rfloor\right\}_p$ for any $p\in[0,1]$. What can be said about the joint distribution of $\big(\RVnumberAP{\ell}{}(\left[\lfloor s n \rfloor\right]_p),\RVnumberAP{\ell}{}(\left[\lfloor t n \rfloor\right]_p)\big)$? More generally, does the random process $\big(\RVnumberAP{\ell}{}(\left[\lfloor tn\rfloor\right]_p)\big)_{t\ge 0}$ satisfy a functional central limit theorem? What about$\big(\RVnumberAP{\lfloor s \ell \rfloor}{}(\left[\lfloor tn\rfloor\right]_p)\big)_{s, t\ge 0}$ for $ \ell=\ell(n) = o(\log n)$? 

%%%%%%%%%%%%%%%%%%%%%%%%%%%%%%%%%%%%%%%%%%%%%%%%%%%%%%%

\end{document}